\documentclass[reqno]{amsart}% {article}

\usepackage{amssymb}
\usepackage{amsmath}
\usepackage{amsthm}
\usepackage{todonotes}
\usepackage{soul}
\usepackage{cancel}

\usepackage{float}

\usepackage{comment}

\usepackage{stmaryrd}
\usepackage{romanbar}

\usepackage{color}

\newtheorem{thm}{Theorem}
\newtheorem{prop}{Proposition}

\newtheorem{cor}{Corollary}

\newcommand{\ve}{\mathbf{e}}
\newcommand{\vf}{\mathbf{f}}

\newcommand{\vn}{\mathbf{n}}
\newcommand{\vu}{\mathbf{u}}
\newcommand{\vv}{\mathbf{v}}
\newcommand{\vw}{\mathbf{w}}
\newcommand{\vx}{\mathbf{x}}

\newcommand{\vz}{\mathbf{z}}

\newcommand{\vvarphi}{\boldsymbol{\varphi}}

\newcommand{\veps}{\boldsymbol{\varepsilon}}

\newcommand{\vel}{\boldsymbol{\ell}}

\newcommand{\vI}{\boldsymbol{I}}

\newcommand{\vT}{\boldsymbol{T}}
\newcommand{\vU}{\boldsymbol{U}}
\newcommand{\vV}{\boldsymbol{V}}

\newcommand{\vX}{\boldsymbol{X}}

\newcommand{\avrg}[1]{{\left\{\kern-0.5ex\left\{ #1 \right\}\kern-0.5ex\right\}}} % average
\newcommand{\tr}{\mbox{tr}} % trace
\newcommand{\di}{\mathop{\rm div}\nolimits} % divergence operator
\newcommand{\cC}{\mathcal{C}}
\newcommand{\cD}{\mathcal{D}}

\newcommand{\cK}{\mathcal{K}}

\newcommand{\cP}{\mathcal{P}}

\newcommand{\pdd}{\partial}
\newcommand{\Om}{\Omega}
\newcommand{\bOm}{\pdd \Om}
\newcommand{\nd} {\text{d}}

\begin{document}

\title[Error analysis for the approximation of a flow in a porous media]{Error analysis for the approximation of a flow in deformable porous media with nonlinear strain-stress relation}

\author{Andrea Bonito}
\address{(Andrea Bonito) Texas A\&M University, Department of Mathematics, College Station, TX 77843, USA}
\email{bonito@tamu.edu}

\author{Vivette Girault}
\address{(Vivette Girault) Sorbonne-Universit\'e, CNRS, Universit\'e de Paris, Laboratoire Jacques-Louis Lions (LJLL), F-75005 Paris, France}
\email{vivette.girault@sorbonne-universite.fr}

\author{Diane Guignard}
\address{(Diane Guignard) University of Ottawa, Department of Mathematics and Statistics, 150 Louis-Pasteur Pvt, Ottawa, ON, Canada K1N 6N5}
\email{dguignar@uottawa.ca}



\begin{abstract}
We study a model describing the slow flow of a fluid through a deformable, porous, elastic solid undergoing small deformations. The stress–strain relationship of the solid incorporates nonlinear effects, formulated as a perturbation of the classical linear elasticity.
To approximate the coupled system, we introduce a discrete scheme based on a first-order semi-implicit time integration scheme combined with a standard finite element spatial discretization. We establish the existence and uniqueness of the discrete solution and derive a priori convergence estimates under the assumption that the nonlinear perturbations remain sufficiently small.
Finally, we demonstrate the efficiency of the proposed scheme through numerical experiments that also highlight the nonlinear phenomena captured by the model.
\end{abstract}

\maketitle

%--------------------------------
\section{Introduction}
 \label{sec:intro}

The most popular models to describe flow through porous media are the celebrated Darcy's equations , which are valid for rigid media and when the fluid flow is driven by diffusion.
Several generalizations of Darcy's equations have been proposed. We mention the work of Brinkman \cite{brinkman1949calculation,brinkman1949permeability} accounting for the diffusion within the fluid and the work of Fick \cite{fick1855ueber} where the diffusivity and/or drag coefficients depend on the deformation.  
These are special cases that can be derived from the theory of mixture \cite{truesdell1957sulle,bowen1976theory,atkin1976continuum,samohyl1987thermodynamics,rajagopal1995mechanics}. In \cite{rajagopal2007hierarchy}, Rajagopal derived a hierarchy of models for the flow of a fluid through a porous medium. Assuming small deformations, i.e., small displacement gradients and that the flow is slow, the balance of linear momentum for the fluid is approximated as in the derivation of the Brinkman's equations. The fluid velocity $\vv$ and the fluid pressure $p$ are related to the porous media deformation $\vu$ by the following relation 
% truesdell1957sulle2,
\begin{equation}
\label{eq:flow}
\alpha (\vv-\pdd_t \vu) - \nu \Delta\, \vv   + \nabla\, p = -\varrho \pdd_t \vv ,\qquad \di\,\vv  = 0.
\end{equation}
Here, $\alpha>0$ is the drag coefficient, $\varrho$ is the density of the fluid, and $\nu$ is the dynamic viscosity of the fluid.
Neglecting the solid acceleration, the balance of linear momentum for the solid relates the solid stress tensor $\vT$ with the same drag term present in the fluid equation
\begin{equation}
\label{eq:bal}
0 = \di(\vT) + \alpha( \vv - \pdd_t \vu).
\end{equation}

Regarding the elastic response of the solid, we consider a special case of the general model introduced in \cite{rajagopal2003implicit,rajagopal2007elasticity,rajagopal2011conspectus,rajagopal2021implicit}.
In particular, we are interested in the constitutive relations used in \cite{arumugam2024new} to
describe the response of bones
\begin{equation}
\label{eq:eps=T}
\veps= E_1 (1+ \lambda_1\tr(\veps))\vT + E_2 (1+ \lambda_2\tr(\veps))\tr(\vT) \vI, 
\end{equation}
where 
$$
\veps:=\veps(\vu):= \frac 1 2 \left(\nabla \vu + (\nabla \vu)^t\right),
$$
the material parameters $E_1$ and $E_2$ satisfy
\begin{equation}
\label{eq:signE}
E_1 >0, \quad E_2 <0, \qquad E_1-d|E_2|>0,
\end{equation}
with $d$ the dimension of the physical domain, and $\lambda_1$, $\lambda_2$ are model parameters. 
When $\lambda_1=\lambda_2=0$, $E_1$ and $E_2$ relates to Lam\'e parameters $\mu$ and $\lambda$ according to
\begin{equation} \label{eqn:Ei_Lame}
    E_1=\frac{1}{2\mu} \qquad \text{and} \qquad E_2=-\frac{\lambda E_1^2}{d\lambda E_1+1}=-\frac{\lambda}{2\mu(d\lambda+2\mu)},
\end{equation}
which leads to the classical linearized elastic model
$$\vT=2\mu\veps(\vu)+\lambda\text{tr}(\veps(\vu))I.$$
Notice the sign assumptions \eqref{eq:signE} are satisfied provided $\mu>0$ and $\lambda\ge 0$.

It is worth mentioning that \eqref{eq:eps=T} follows from the linearization of the constitutive relation for the strain, which is justified under the small deformation assumption.  However, the linearized strain remains expressed as a nonlinear function of the stress. Unlike the classical linearized elastic constitutive relation, it was observed in \cite{arumugam2024new} that the response to tension is noticeably different than its response to compression in uniaxial problems, and the moduli of the material depend on the density of the material. These effects are amplified when $\lambda_1$ and $\lambda_2$ increase and what differentiates the current work from \cite{GGHR1}, where an analysis of the model is provided after  \eqref{eq:eps=T} is linearized (formally $\lambda_1=\lambda_2=0$).  

The system of equations \eqref{eq:flow}--\eqref{eq:eps=T} are defined on a  bounded connected open set $\Omega \subset \mathbb{R}^d$, $d=2,3$, with Lipschitz continuous boundary $\pdd \Omega$. They are thus complemented by boundary and initial conditions. To simplify the theoretical analysis, we shall only study the case when both the solid and the fluid are subjected to homogeneous Dirichlet boundary conditions:
\begin{equation}
    \label{eq:Dir0}
\vu|_{\bOm} ={\bf 0}, \quad \vv|_{\bOm} ={\bf 0}.     
\end{equation}
The initial conditions are
\begin{equation}
    \label{eq:init}
\vu(0) = \vu_{0}, \quad  \vv(0) =  \vv_{0},  
\end{equation}
with $\vu_{0}$ and $\vv_{0}$ satisfying the boundary conditions \eqref{eq:Dir0} and $\di\,\vv_{0} =0$.

In the linearized case $\lambda_1=\lambda_2=0$, the existence of a solution is studied in \cite{GGHR1}. However, due to the lack of compactness and monotonocity, the system of equations \eqref{eq:flow}-\eqref{eq:eps=T} together with the boundary and initial conditions \eqref{eq:Dir0}, \eqref{eq:init} presents a great mathematical challenge when $|\lambda_1|+|\lambda_2|\not = 0$. This work does not address the  existence and uniqueness of a solution to the nonlinear problem. Rather, we  assume that  $\lambda_1$ and $\lambda_2$ are small relative to $\tr(\veps)^{-1}$ and that the nonlinear problem has a unique solution that satisfies the regularity condition specified in \eqref{additional_reg} below. In turn, we propose and analyze a discrete approximation of \eqref{eq:flow}--\eqref{eq:eps=T}.

In Section~\ref{subsec:Nota}, we make the assumption on the exact solution more precise by describing a variational formulation of \eqref{eq:flow}--\eqref{eq:eps=T} along with the associated functional spaces. The proposed discretizations are detailed in Section~\ref{sec:DiscreteGen}, where we start by noting that for small deformations, the nonlinear relation \eqref{eq:eps=T} is equivalent to the one where the terms $\di \vu$ are truncated. We then propose a standard finite element discretization of the resulting system coupled with a first order time integration scheme. The existence and uniqueness of the discrete solution is established in Proposition~\ref{pro:existuhn}, while its stability is the subject of Theorem~\ref{thm:bddvh} (for the approximate fluid velocity), Corollary~\ref{cor:bdduhn} (for the approximate solid displacement), and Proposition~\ref{pro:bddpres} (for the discrete pressure). The truncation operation is critical to derive these results as well as to obtain the a priori error estimates guaranteed by Theorem~\ref{thm:velerr} in Section~\ref{sec:Errorbouds}. While the stability is described for a general finite element method, the a priori error analysis is performed for a first order method only to simplify the discussion.  In Section~\ref{sec:num_res}, we present numerical illustrations of the impact of nonlinearities in the model.

\subsection*{Dedication and Acknowledgments}

This work is dedicated to the memory of our dear friend and colleague Kumbakonam Rajagopal, who initiated this work with the co-authors.  We also would like to thank Dr. Shriram Srinivasan for the precious discussions about the model as well as for commenting on the numerical experiments. 

This research was supported by the NSF Grant DMS-2409807 (AB) and the NSERC Discovery Grant RGPIN-2021-04311 (DG).

%--------------------------------
\section{Preliminaries}
\label{subsec:Nota}

Let $T>0$ be a final time and for $0<t\le T$, define the space time domain $Q_t:=\Om \times ]0,t[$
and denote by $\partial_t$ the derivative with respect to $t\in ]0,T[$. In the sequel, if there is no ambiguity, the differential notation in integrals will be omitted.

\subsection{Standard Functional Spaces}

We use $\mathcal D$ to represent $\Omega$ or $\partial \Omega$. 
The scalar product of $L^2(\mathcal D)$ is denoted by $(\cdot,\cdot)_{\mathcal D}$:
$$\forall\, f,g \in L^2(\mathcal D),\quad (f,g)_{\mathcal D} := \int_{\mathcal D} f(\vx) g(\vx) \nd\vx,$$
and the index $\mathcal D$ is omitted when the integration domain is clear from the context. For any positive real number $s$, we denote by $H^{s}(\mathcal D)$ the classical Hilbert space with norm $\|.\|_{H^s(\mathcal D)}$. The reader is referred to \cite{adams2003sobolev,grisvard2011elliptic,leoni2017first} for more details.
Furthermore, for vector-valued functions
$\vv \in H^d$, $d\geq 1$, with $H$ a Hilbert space on $\Omega$, we recall that $H^d$ is a Hilbert space equipped with the norm
$$
\forall\, \vw \in H^d, \qquad \| \vw \|_{H}:=\| \vw \|_{H^d}:= \| |\vw|\|_H,
$$
where $|\cdot|$ is the Frobenius norm. Similarly, we set 
$$
\forall\, \vw \in H^{d\times d}, \qquad \| \vw \|_{H}:=\| \vw \|_{H^{d\times d}}:= \| |\vw|\|_H.
$$

The importance of fractional order spaces stems from the following trace property: If $v$ belongs to $H^s(\Om)$ for some $s \in ]\frac{1}{2},1]$,
 then its trace on $\pdd \Om$ belongs to $H^{s-\frac{1}{2}}(\pdd\Om)$ and there exists a constant  $C_s$  such that
\begin{equation}
\label{eq:trace.s}
\forall\, v \in H^s(\Om), \qquad \|v\|_{H^{s-\frac{1}{2}}(\pdd\Om)} \le  C_s  \|v\|_{H^s(\Om)}.
\end{equation}
In particular,  $H^{\frac{1}{2}}(\pdd\Om)$ is the trace space of $H^1(\Om)$ and we denote by $H^{-\frac{1}{2}}(\pdd\Om)$ the dual space of $H^{\frac{1}{2}}(\pdd\Om)$.

Let $H^1_0(\Om)$ be the set of functions in $H^1(\Om)$ that vanish on $\pdd\Om$.
We denote by $H^{-1}(\Omega)$ its dual and recall that the following Poincar\'e and Korn inequalities, written for all functions $\vw$ in $H^1_0(\Om)^d$, hold
\begin{equation}
\|\vw\|_{L^2(\Om)}\le  \cP  |\vw|_{H^1(\Om)} ,
\label{eq:Poincare}
\end{equation}
\begin{equation}
|\vw|_{H^1(\Om)}\le  \cK  \|\veps(\vw)\|_{L^2(\Om)} ,
\label{eq:Korn}
\end{equation}
where $\cP$ and $\cK $ are positive constants depending only on $\Omega$. 

As usual, for handling time-dependent problems, it is convenient to consider measurable functions defined on a time interval $]a,b[$ with values in a Banach space, say $(X,\|.\|_X)$. More precisely, for any number $r$, $1\le r\le \infty$, we define the Banach space
$$
L^r(a,b;X) := \{f \mbox{ measurable in } ]a,b[\,;
 \int_a^b \|f(t)\|_X^r\nd t <\infty\},
$$ 
equipped with the norm
$$\|f\|_{L^r(a,b;X)} := \left(\int_a^b \|f(t)\|_X^r\,\nd t\right)^{\frac{1}{r}},$$
with the usual modification if $r = \infty$. It is a Hilbert space if $X$ is a Hilbert space. In particular,
$$L^2(a,b;L^2(\Om)) = L^2(\Om \times ]a,b[).
$$

We shall also need the space-time functions that are continuous in time
$$
\cC^0(a,b;X) := \{f \in L^\infty(a,b;X) \,; \ t \mapsto \| f(t) \|_X \quad \mbox{is continuous on }[a,b]\}.
$$ 

\subsection{Variational Formulation}

We observe that from the point of view of mathematics, $\vT$ can be written explicitly in terms of $\veps(\vu)$. Indeed, by taking the trace of both sides of \eqref{eq:eps=T}, using that $\tr(\veps(\vu))=\di\,\vu$, and collecting terms we obtain
\begin{equation}
\label{eq:tr(1)}
\di\, \vu = \tr(\vT) \Big(E_1 (1+ \lambda_1 \di\, \vu) - 3|E_2| (1+ \lambda_2 \di\, \vu)\Big).
\end{equation}
To simplify, we introduce the notation
\begin{equation}
\label{eq:F(us)}
F(f) =E_1 (1+ \lambda_1f) - d|E_2| (1+ \lambda_2f)
\end{equation}
so that
\begin{equation}
\label{eq:tr(Ts)}
\tr(\vT) = \frac{1}{F(\di\, \vu) }\di\,\vu,
\end{equation}
and substituting into \eqref{eq:eps=T}, we obtain the equivalent relation
\begin{equation}
\label{eq:(Ts)}
\vT = \frac{1}{E_1 (1+ \lambda_1 \di\, \vu) } \Big( \veps(\vu) + |E_2| (1+ \lambda_2 \di\, \vu) \frac{\di\,\vu}{F(\di\,\vu)} \vI\Big).
\end{equation}

The duality product of \eqref{eq:bal} with a test function $\vz$ in $H^1_0(\Om)^d$ together with the expression \eqref{eq:(Ts)} for $\vT$ yield 
\begin{equation}
\label{eq:varbal2}
\begin{split}
\int_\Om \frac{1}{E_1 (1+ \lambda_1 \di\, \vu)} \veps( \vu): \veps( \vz) &+  \int_\Om \frac{|E_2| (1+ \lambda_2 \di\, \vu)}{E_1 (1+ \lambda_1\di\, \vu)}\frac{\di\,\vu}{F(\di\,\vu)} \di\,\vz \\
& 
+ \alpha \langle\pdd_t \vu,\vz \rangle
 = \alpha \int_\Om  \vv \cdot \vz.
\end{split}
\end{equation}

For each $\vv$ given in $L^2(Q_T)^d$, when integrated with respect to time over $]0,t[$, for $t \in ]0,T]$, and when $\lambda_1=\lambda_2=0$, \eqref{eq:varbal2} is a standard linear elasticity system that is well defined for test functions in $L^2(0,T;H^1_0(\Om)^d)$. Therefore, it is reasonable to seek $\vu \in  W(0,T)$, where
\begin{equation}
\label{eq:W0T}
W(0,T) := \{ \vw \in L^2(0,T; H^1_0(\Om)^d); \quad \pdd_t \vw \in L^2(0,T; H^{-1}(\Om)^d)\}
\end{equation}
is a Hilbert space equipped with the graph norm 
$$
\|\vw\|_{W} := \big(\|\vw\|^2_{L^2(0,T; H^1_0(\Om)^d)} + \|\pdd_t \vw\|^2_{L^2(0,T; H^{-1}(\Om)^d)}\big)^{\frac{1}{2}}.
$$
We note that the functions in $W(0,T)$ are continuous in time with values in $L^2(\Om)^d$, i.e., for $\vw \in W(0,T)$, there is a function in the same equivalence class (still denoted $\vv$) such that the mapping $t \mapsto \| \vw(t)\|_{L^2(\Om)}$ is continuous on $[0,T]$.
We denote by $C_L$ the corresponding embedding constant
\begin{equation}
    \label{eq:CL}
\forall\, \vw \in W(0,T),\quad  \sup_{t \in [0,T]}\|\vw(t)\|_{L^2(\Om)} \le C_L \|\vw\|_W.
\end{equation}

For the fluid velocity-pressure, we note that for a given $\pdd_t \vu \in L^2(0,T;H^{-1}(\Omega)^d)$, \eqref{eq:flow} is a Stokes-like system. 
We consider the standard functional space
\begin{equation}
\label{eq:V}
V := \{ \vw \in H^1_0(\Omega)^d\,;\, \di\, \vw = 0\quad \text{ in }\quad \Omega\},
\end{equation}
and its dual space $V^\prime$ equipped with the dual norm
\begin{equation}
\label{eq:V'}
\|\vel\|_{V^\prime} := \sup_{\vw \in V}\frac{\langle \vel, \vw \rangle}{|\vw|_{H^1(\Om)}}.
\end{equation}
We then set 
\begin{equation}
\label{eq:WV0T}
W_V(0,T) := \{ \vw \in L^2(0,T; V); \quad \pdd_t \vw \in L^2(0,T; V^\prime)\},
\end{equation}
which is again a Hilbert space equipped with the graph norm
$\|\cdot\|_{W_V}$. 
As for the displacement space $W(0,T)$, functions in $W_V(0,T)$ are continuous in time and we denote by $C_V$ the embedding constant, i.e.,
\begin{equation}
    \label{eq:CWV}
\forall\, \vw \in W_V(0,T),\quad  \sup_{t \in [0,T]}\|\vw(t)\|_{L^2(\Om)} \le C_V \|\vw\|_{W_V}.
\end{equation}
We shall also often need the functional space $L^2(0,T; H^{-1}(\Om)^d)$ and we denote its norm by
$$\|\vel\|_{*} := \|\vel\|_{L^2(0,T; H^{-1}(\Om)^d)}.$$ 

The pressure is determined up to a constant, which led to consider the subspace of $L^2(\Om)$ with vanishing mean value
\begin{equation}
\label{eq:L20}
L^2_0(\Om) := \{ f \in L^2(\Omega)\,;\,\int_\Om f (\vx)\,\nd\vx  = 0\}.
\end{equation}
To account for the dependency in time, we define 
$$
W^{1,1}(0,T;L^2_0(\Om)) := \{f \in L^1(0,T;L^2_0(\Om))\,;\, \partial_t\,f \in L^1(0,T;L^2_0(\Om))\} \subset \cC^0(0,T;L^2_0(\Omega)),
$$
$W^{1,1}_0(0,T;L^2_0(\Om))$ its restriction to functions vanishing at $t=0$ and $t=T$, 
and let $W^{-1,\infty}(0,T;L^2_0(\Om))$ its dual space where the pressure will be sought.

Now, we take the duality product of the first part of
\eqref{eq:flow} with a test function $\vw$ in $H^1_0(\Om)^d$ and seek $\vv(t)$ in $V$ for a.e. $t\in ]0,T]$ such that 
\begin{equation}
\label{eq:varbal3}
\langle \varrho \pdd_t \vv + \nabla\, p, \vw\rangle + \nu \int_\Om \nabla\, \vv : \nabla\,
\vw + \alpha \int_\Om \vv \cdot \vw
= \alpha \langle\pdd_t \vu,\vw \rangle.  
\end{equation}
It is well-known that the velocity  $\vv$ belongs to $W_V(0,T)$, $p$ to $W^{-1, \infty}(0,T;L^2_0(\Om))$ \cite{temam1977navier,boyer2012mathematical} and  the sum $\pdd_t \vv + \frac 1 {\varrho}\nabla\, p$ belongs to $L^2(0,T;H^{-1}(\Om)^d)$ provided $\pdd_t \vu \in L^2(0,T;H^{-1}(\Om)^d)$. This suggests considering
\begin{equation}
    \label{bbX}
\mathbb{X} := \{(\vw,q) \in W_V(0,T) \times  W^{-1, \infty}(0,T;L^2_0(\Om))\,;\, \pdd_t \vw + \frac{1}{\varrho} \nabla\,q \in L^2(0,T;H^{-1}(\Om)^d)\}
\end{equation}
as the solution space for the fluid velocity so that when integrated with respect to time over $]0,t[$, for $t \in ]0,T]$,  \eqref{eq:varbal3} is  well defined for test functions $\vw$ in $L^2(0,T;H^1_0(\Om)^d)$.
The space $\mathbb X$ 
equipped with the norm
\begin{equation}
\label{normbbV}
\|(\vw,q)\|_{\mathbb{X}} = \| \vw\|_{W_V(0,T)} + \|q\|_{W^{-1, \infty}(0,T;L^2_0(\Om))} + \|\pdd_t \vw + \frac{1}{\varrho} \nabla\,q \|_*
\end{equation}
is a Banach space.

Gathering the above consideration, the variational formulation consists of finding $\big(\vu, \vv,p\big) \in W(0,T) \times \mathbb X$ satisfying
\begin{equation}\label{variational}
\eqref{eq:varbal2} \quad \forall\, \vz\in L^2(0,T;H^1_0(\Om)^d), \qquad  \eqref{eq:varbal3} \quad \forall\, \vw \in L^2(0,T;H^1_0(\Om)^d),
\end{equation}
and the initial conditions $ \vu(0)=\vu_{0}$ and $\vv(0)=\vv_{0}$.

We do not discuss the well-posedness of the variational formulation, but mention that a major obstacle in guaranteeing the well-posedness of the variational formulation for any data is the lack of exploitable compactness  in \eqref{eq:varbal2} and \eqref{eq:varbal3} and/or that the nonlinear relation \eqref{eq:(Ts)} is not monotone in $\veps(\vu)$.
However, when $\lambda_1$ and $\lambda_2$ are small, the system considered can be viewed as a small perturbation of
the linear model (i.e., with $\lambda_1=\lambda_2=0$) which is itself
well-posed.

In this work, we assume that for some $\bar \lambda>0$, the variational formulation has a solution for all $|\lambda_1|, |\lambda_2|\leq \bar \lambda$ and that
\begin{equation}\label{additional_reg}
\begin{split}
&\vv \in L^\infty(0,T;H^2(\Omega)^d), \qquad \partial_t \vv \in   L^2(0,T;H^1(\Omega)^d), \\
&\vu \in L^\infty(0,T;H^2(\Omega)^d), \qquad \partial_t \vu \in  L^2(0,T;H^1(\Omega)^d), \qquad \partial^2_{tt} \vu \in L^2(0,T;H^{-1}(\Omega)^d), \\
& p \in L^{\infty}(0,T;H^1(\Omega)),
\end{split}
\end{equation}
(with implicit dependence on $\lambda_1$ and $\lambda_2$).
In addition, we suppose that there is a constant $D$ such that
\begin{equation}\label{e:small_trace}
\| \di\, \vu \|_{L^\infty(0,T;L^\infty(\Omega)^d)} \leq D, \qquad \forall |\lambda_1|,|\lambda_2|\leq \bar \lambda.
\end{equation}
Observe that the above assumption fits the small strain assumption considered in this work thanks to $| \di\, \vu| \leq \sqrt d |\veps(\vu)|$.

%--------------------------------
\section{Discrete Scheme}
 \label{sec:DiscreteGen}

\subsection{Truncation}
\label{sec:trunc}

We start by noting that \eqref{e:small_trace} implies that the expression \eqref{eq:(Ts)}  for $\vT$ is equivalent to
\begin{equation}
\label{eq:(Ts_trunc)}
\vT = \frac{1}{E_1 (1+ \lambda_1 T_{\delta}\di\, \vu) } \Big( \veps(\vu) + |E_2| (1+ \lambda_2T_{\delta} \di\, \vu) \frac{\di\,\vu}{F(T_{\delta} \di\,\vu)} \vI\Big),    
\end{equation}
where $\delta \geq D$ is fixed and for $r>0$, $T_r$ is the truncation operator defined by

\begin{equation}
\label{e:tunc}
T_r f = f \quad \mbox{if}\ |f| \le r, \quad T_r f = r\,{\rm sign}(f) \quad \mbox{if}\ |f| >r.
\end{equation}
Note that there holds
$$T_{\delta} \di(\vv + \vw) - T_{\delta} \di\,\vv = \begin{cases}
\di\, \vw & \text{if} \ -\delta \le \di(\vv + \vw) \le \delta\ \text{and}\ -\delta \le \di\, \vv \le \delta,  \\
 0 & \text{if} \ \di(\vv + \vw) > \delta\ \text{and}\ \di\, \vv  > \delta,\\
 & \text{or}\ \di(\vv + \vw) < - \delta\ \text{and}\ \di\, \vv < -\delta.
  \end{cases}
$$
and thus
\begin{equation}
\label{eq:difTdelt}
 |T_{\delta} \di(\vv + \vw) - T_{\delta} \di\,\vv | \le | T_{2\delta} \di\,\vw| \qquad \forall\, \vv,\vw\in H^1(\Om)^d.
\end{equation}
With the modification \eqref{eq:(Ts_trunc)}, the variational formulation \eqref{eq:varbal2} becomes for $\vz \in H^1_0(\Omega)^d$
\begin{equation}
\label{eq:varbal2bis}
\begin{split}
\int_\Om \frac{1}{E_1 (1+ \lambda_1 T_{\delta}\di\, \vu)} \veps( \vu): \veps( \vz) &+  \int_\Om \frac{|E_2| (1+ \lambda_2 T_{\delta}\di\, \vu)}{E_1 (1+ \lambda_1T_{\delta}\di\,\vu})\frac{\di\,\vu}{F(T_{\delta}\di\,\vu)} \di\,\vz \\
& 
+ \alpha \langle\pdd_t \vu,\vz \rangle
 = \alpha \int_\Om  \vv \cdot \vz.
\end{split}
\end{equation}

Although the truncation operator does not have any effect at the continuous level, it will be instrumental to derive the stability of the discrete scheme. In particular, we now derive uniform estimates for the coefficients appearing in the variational equation \eqref{eq:varbal2bis}:
$$
\frac{1}{1+ \lambda_1T_{\delta} \di\ \vw}, \qquad \frac{1}{F(T_{\delta} \di\ \vw)}, \qquad \textrm{and}\quad 1+ \lambda_2T_{\delta} \di\, \vw
$$
for some $\vw \in H^1_0(\Omega)^d$.
In order to help with the presentation, we introduce a few constants.
Recall that we are interested in $\lambda_1$, $\lambda_2$ such that $|\lambda_1|, |\lambda_2| \leq \bar \lambda$ so that \eqref{e:small_trace} holds.
We set $\gamma_1 := \frac{1}{1+ |\lambda_1| \delta}>0$ and
$\gamma_2 := \frac{1}{1- |\lambda_1| \delta}$ so that for $|\lambda_1|$ sufficiently small, $1-|\lambda_1|\delta>0$ and thus
\begin{equation}\label{e:gam12}
\gamma_1 \le \frac{1}{1+ \lambda_1T_{\delta} \di\ \vw} \leq \gamma_2.
\end{equation}
We also set $\gamma_5 := \frac{1}{\gamma_3 + |\gamma_4| \delta}>0$ and $\gamma_6 := \frac{1}{\gamma_3- |\gamma_4| \delta}$, where $\gamma_3$ and $\gamma_4$ are defined by 
\begin{equation}\label{eq:g3_4}
\gamma_3:=E_1+dE_2 \qquad  \textrm{and} \qquad \gamma_4:=E_1\lambda_1+dE_2\lambda_2.
\end{equation}
With these definitions, we have $F(T_{\delta} \di \vw)= \gamma_3 + \gamma_4 T_{\delta} \di \vw$, so that for $|\lambda_1|$, $|\lambda_2|$ sufficiently small, we get $\gamma_3-|\gamma_4|\delta>0$ and thus
\begin{equation}
 \label{eq:gam56.1}
     \gamma_5 \le \frac{1}{F(T_{\delta} \di\, \vw)} \le \gamma_6.
 \end{equation}
Note that we have used here the assumption $\gamma_3>0$ and $E_2<0$, see \eqref{eq:signE}.  Finally, we set $\gamma_7 := 1- |\lambda_2| \delta$ and $\gamma_8 := 1+ |\lambda_2| \delta$ so that for $|\lambda_2|$ sufficiently small
\begin{equation}
 \label{eq:gam78.1}
     0<\gamma_7 \le 1+ \lambda_2T_{\delta} \di\, \vw \le \gamma_8.
 \end{equation}

Notice that the larger $\delta$ is, the smaller $|\lambda_1|$ and $|\lambda_2|$ are required to be.

\subsection{General finite element discretization}

From now on, we suppose that $\Om$ is a Lipchitz polytopal domain. Let $\vU_h$ be a finite element subspace of $H^1_0(\Om)^d$ designed to approximate the displacement, $\vX_h$ a finite element subspace of $H^1_0(\Om)^d$ designed to approximate the fluid's velocity and $Q_h$ a finite element subspace of $L^2_0(\Om)$ designed to approximate the fluid's pressure. We assume that $\vX_h$ and $Q_h$ satisfy a uniform inf-sup condition for the divergence, i.e., there is 
$\beta^*>0$ independent of $h$ such that
\begin{equation}\label{e:discrete_infsup}
\beta^*:= \inf_{q \in Q_h} \sup_{\vw \in \vX_h} \frac{\int_\Om q \di\,\vw}{\| q \|_{L^2(\Om)} \| \nabla \vw\|_{L^2(\Om)}}>0.
\end{equation}
 
As usual the divergence-free functions are approximated by the space $\vV_h$ defined by
$$\vV_h = \{\vv_h \in \vX_h\,;\, \int_\Om q_h \di\, \vv_h = 0, \forall\, q_h \in Q_h\}.$$
To avoid a multiplicity of notation, we use the same symbol $\Pi_h$ for an interpolation operator  from $H^1_0(\Om)^d$ into $\vU_h$, from $H^1_0(\Om)^d$ into $\vX_h$, and from $\vV$ into $\vV_h$. We suppose that $\Pi_h$ is stable in $H^1_0(\Om)^d$, i.e., there exists a positive constant $C$ independent of $h$ such that for all $\vu \in H^1_0(\Om)^d$
\begin{equation}
\label{eq:PihC}
 \|\nabla\,\Pi_h\vu\|_{L^2(\Om)} \le C \|\nabla\,\vu\|_{L^2(\Om)} \qquad   \textrm{and} \qquad  \|\vu-\Pi_h\vu\|_{L^2(\Om)} \le Ch\|\nabla\,\vu\|_{L^2(\Om)}.
\end{equation}

\subsection{Fully Discrete Scheme}
 
Let the number of time steps $N>1$ be chosen, set $\Delta\,t := \frac{T}{N}$ and $t_n:=n\Delta\,t$, $0\le n\le N$, the snapshots in time. Then, starting from
\begin{equation}
\label{eq:initial}
 \vu_h^0 = \Pi_h\vu_{0,s}, \quad \vv_h^0 = \Pi_h\vv_{0,f},
\end{equation}
the relations  \eqref{variational} with \eqref{eq:varbal2} replaced by \eqref{eq:varbal2bis}, are discretized as follows: given $\vu_h^{n-1} \in \vU_h$ and $\vv_h^{n-1} \in \vV_h$ find $\vu_h^{n} \in \vU_h$, $\vv_h^{n} \in \vV_h$, and $p_h^{n} \in Q_h$ satisfying for all $\vz_h \in \vU_h$,
\begin{equation}
\label{eq:displachn}
 \begin{split}
 &\frac{\alpha}{\Delta\,t}\int_\Om (\vu_h^n- \vu_h^{n-1})\cdot \vz_h + \frac{1}{E_1}\int_\Om \frac{1}{1+ \lambda_1T_{\delta} \di\, \vu_h^{n-1}} \veps( \vu_h^n): \veps( \vz_h) \\
 & + \frac{|E_2|}{E_1} \int_\Om \frac{ 1+ \lambda_2T_{\delta} \di\, \vu_h^{n-1}}{( 1+ \lambda_1T_{\delta} \di\, \vu_h^{n-1})}
\frac{1}{F(T_{\delta} \di \vu_h^{n-1})}
(\di\,\vu_h^n) (\di\,\vz_h)  \\
 & = \alpha \int_\Om  \vv_h^n \cdot \vz_h ,
 \end{split}
\end{equation}
and for all $\vw_h \in \vX_h$,
\begin{equation}
\label{eq:fluidhn}
 \begin{split} 
 &  \frac{\varrho}{\Delta\,t}\int_\Om (\vv_h^n- \vv_h^{n-1})\cdot \vw_h
 - \int_\Om p_h^n\di\,\vw_h + \nu \int_\Om \nabla\, \vv_h^n : \nabla\,
\vw_h  + \alpha \int_\Om \vv_h^n \cdot \vw_h \\
 & = \frac{\alpha}{\Delta\,t}\int_\Om (\vu_h^n- \vu_h^{n-1})\cdot\vw_h.
\end{split}
\end{equation}
At each step $n$, \eqref{eq:displachn}--\eqref{eq:fluidhn} is a coupled linear system in finite dimension.

%--------------------------------
\subsection{Existence and uniqueness of a discrete solution}
\label{subsec:exist.h}

For fixed $n\ge 1$, let us check that given $\vu_h^{n-1} \in \vU_h$ and $\vv_h^{n-1} \in \vV_h$, the system \eqref{eq:displachn}--\eqref{eq:fluidhn} defines $\vu_h^n$, $\vv_h^n$, and $p_h^n$. Let
\begin{equation} \label{eqn:B1B2}
    B_1 := \frac{1}{1+ \lambda_1T_{\delta} \di\, \vu_h^{n-1}}, \quad B_2 := \frac{ 1+ \lambda_2T_{\delta} \di\, \vu_h^{n-1}}{( 1+ \lambda_1T_{\delta} \di\, \vu_h^{n-1})}
\frac{1}{F(T_{\delta} \di\, \vu_h^{n-1})}.
\end{equation}
Notice that both constants are well defined and positive thanks to  \eqref{e:gam12}, \eqref{eq:gam56.1}, and  \eqref{eq:gam78.1} provided that $|\lambda_1|, |\lambda_2|$ are sufficiently small.

Then for this $n$,  \eqref{eq:displachn} reads
\begin{equation}
\label{eq:displachn2}
\begin{split}
    \frac{\alpha}{\Delta\,t}\int_\Om (\vu_h^n- \vu_h^{n-1})\cdot \vz_h + \frac{1}{E_1}\int_\Om B_1 \veps( \vu_h^n): &\veps( \vz_h) 
 + \frac{|E_2|}{E_1} \int_\Om B_2
(\di\,\vu_h^n) (\di\,\vz_h)\\
&
 = \alpha \int_\Om  \vv_h^n \cdot \vz_h.
 \end{split}
\end{equation}
Since $B_1$ and $B_2$ are known, \eqref{eq:fluidhn}--\eqref{eq:displachn2} is a linear system in $(\vu_h^n,\vv_h^n, p_h^n)$  and establishing the existence of a solution is equivalent to proving that if the data $\vu_h^{n-1}$ and $\vv_h^{n-1}$ are both zero then the only solution $(\vu_h^n,\vv_h^n, p_h^n)$ is the zero solution. 
\begin{prop}
\label{pro:existuhn}
Let $\bar \lambda>0$ be as in \eqref{e:small_trace} and assume that $|\lambda_1|, |\lambda_2| \leq \bar \lambda$ are sufficiently small so that \eqref{e:gam12}, \eqref{eq:gam56.1},  and \eqref{eq:gam78.1} hold. Then the system \eqref{eq:displachn}--\eqref{eq:fluidhn} has exactly one solution at each step $n$.
\end{prop}

\begin{proof}
Let $(\vu_h^n,\vv_h^n, p_h^n)$ be any solution of \eqref{eq:displachn}--\eqref{eq:fluidhn} with data $\vu_h^{n-1}={\bf 0}$ and $\vv_h^{n-1} = {\bf 0}$. Choosing $\vz_h = \vu_h^n $ and $\vw_h = \vv_h^n$ gives
\begin{equation}
\label{eq:displachn3}
\frac{\alpha}{\Delta\,t}\|\vu_h^n\|^2_{L^2(\Om)} + \frac{1}{E_1}\int_\Om B_1 |\veps( \vu_h^n)|^2
 + \frac{|E_2|}{E_1} \int_\Om B_2
|\di\,\vu_h^n|^2
= \alpha \int_\Om  \vv_h^n \cdot \vu_h^n,
\end{equation}
and
$$
\frac{\varrho}{\Delta\,t}\|\vv_h^n\|^2_{L^2(\Om)}
 +  \nu \| \nabla\, \vv_h^n\|^2_{L^2(\Om)}  + \alpha \|\vv_h^n\|^2_{L^2(\Om)}
= \frac{\alpha}{\Delta\,t}\int_\Om \vu_h^n\cdot\vv_h^n.
$$
By combining the two right-hand sides, we immediately derive
\begin{equation}
\label{eq:eq:fluidhn2}
\begin{split}
\frac{\varrho}{\Delta\,t}&\|\vv_h^n\|^2_{L^2(\Om)}
 +  \nu \| \nabla\, \vv_h^n\|^2_{L^2(\Om)}  + \alpha \|\vv_h^n\|^2_{L^2(\Om)}\\
&= \frac{\alpha}{\Delta\,t^2}\|\vu_h^n\|^2_{L^2(\Om)} + \frac{1}{\Delta\,t E_1}\int_\Om B_1 |\veps( \vu_h^n)|^2 
 + \frac{|E_2|}{ \Delta\,tE_1} \int_\Om B_2
|\di\,\vu_h^n|^2.
\end{split}
\end{equation}
But \eqref{eq:displachn3} also implies
\begin{equation}
\label{eq:displachn4}
\frac{\alpha}{2\Delta\,t}\|\vu_h^n\|^2_{L^2(\Om)} + \frac{1}{E_1}\int_\Om B_1 |\veps( \vu_h^n)|^2 
 + \frac{|E_2|}{E_1} \int_\Om B_2
|\di\,\vu_h^n|^2 \le \frac{\alpha}{2}\Delta\,t  \|\vv_h^n\|^2_{L^2(\Om)};
\end{equation}
as the functions $B_1$ and $B_2$ are positive, this implies in particular
\begin{equation}
\label{eq:uhvh}
\frac{1}{\Delta\,t} \|\vu_h^n\|^2_{L^2(\Om)} \le \Delta\,t \|\vv_h^n\|^2_{L^2(\Om)}.
\end{equation}
Adding \eqref{eq:displachn4} and \eqref{eq:uhvh} multiplied by $\frac{\alpha}{2}$ gives
$$
\frac{\alpha}{\Delta\,t}\|\vu_h^n\|^2_{L^2(\Om)} + \frac{1}{E_1}\int_\Om B_1 |\veps( \vu_h^n)|^2 
 + \frac{|E_2|}{E_1} \int_\Om B_2
|\di\,\vu_h^n|^2 \le \alpha \Delta\,t \|\vv_h^n\|^2_{L^2(\Om)}
$$
so that substituting  into \eqref{eq:eq:fluidhn2}, we obtain
$$
\frac{\varrho}{\Delta\,t}\|\vv_h^n\|^2_{L^2(\Om)}
 +  \nu \| \nabla\, \vv_h^n\|^2_{L^2(\Om)}  + \alpha \|\vv_h^n\|^2_{L^2(\Om)} \le \alpha \|\vv_h^n\|^2_{L^2(\Om)}
$$
and thus $\vv_h^n ={\bf 0}$. Returning to  \eqref{eq:uhvh}, we deduce that $\vu_h^n ={\bf 0}$ and so $p_h^n =0$ as well.
\end{proof}

%--------------------------------
\subsection{Stability of the discrete solution}
\label{subsec:stab.h}

In order to derive a priori upper bounds for the discrete solution, we apply the Poincar\'e \eqref{eq:Poincare}, Korn \eqref{eq:Korn} and Young inequalities to  \eqref{eq:displachn} with $\vz_h=\vu_h^n$ and obtain
\begin{equation}
\label{eq:uhn1}
\begin{split}
&\frac{\alpha}{\Delta\,t}\big(\|\vu_h^n\|^2_{L^2(\Om)} - \|\vu_h^{n-1}\|^2_{L^2(\Om)} + \|\vu_h^n- \vu_h^{n-1} \|^2_{L^2(\Om)}\big)\\
&+ \frac{\gamma_1}{E_1} \big(\|\veps(\vu_h^n)\|^2_{L^2(\Om)} + 2|E_2|\gamma_5 \gamma_7 \|\di\,\vu_h^n\|^2_{L^2(\Om)}\big)
 \le \frac{E_1}{\gamma_1}(\alpha \cP \cK)^2\|\vv_h^n\|^2_{L^2(\Om)}.
\end{split}
\end{equation}
Similarly, but  to \eqref{eq:fluidhn} with $\vw_h=\vv_h^n$, we get
\begin{equation}
\label{eq:vhn1}
    \begin{split}
 \frac{\varrho}{\Delta\,t}&\big(\|\vv_h^n\|^2_{L^2(\Om)} - \|\vv_h^{n-1}\|^2_{L^2(\Om)} + \|\vv_h^n- \vv_h^{n-1} \|^2_{L^2(\Om)}\big)\\
& +  \nu \| \nabla\, \vv_h^n\|^2_{L^2(\Om)}  + 2\alpha \|\vv_h^n\|^2_{L^2(\Om)}
  \le \frac{\alpha^2}{\nu }\|\frac{1}{\Delta\,t} (\vu_h^n- \vu_h^{n-1}) \|^2_{H^{-1}(\Om)}.
        \end{split}
\end{equation}

The next proposition gives
an estimate for the discrete time derivative of $\vu_h^n$. 
To simplify the expressions below, we introduce
two constants
\begin{equation}
\label{eq:cD1}
\cD_1 := \frac{(\gamma_6 \gamma_8)^2}{ \gamma_5 \gamma_7} |E_2|, \quad \cD_2 := 1+ \frac{\cD_1}{2},
\end{equation}
and the following quantity for any function $\vf$ in $H^1(\Om)^d$, 
\begin{equation}
\label{eq:d1f}
d_2(\vf) := \|\veps(\vf)\|^2_{L^2(\Om)} + 2|E_2|\gamma_5 \gamma_7  \|\di\,\vf\|^2_{L^2(\Om)}.
\end{equation}

\begin{prop}
\label{pro:duh-1}
Let $\bar \lambda>0$ be as in \eqref{e:small_trace} and assume that $|\lambda_1|, |\lambda_2| \leq \bar \lambda$ are sufficiently small so that \eqref{e:gam12}, \eqref{eq:gam56.1},  and \eqref{eq:gam78.1} hold. Then the  discrete time derivative of $\vu_h^n$ satisfies
\begin{equation}
\label{eq:uhn2-1.1}
\begin{split}
\alpha^2 \|\frac{1}{\Delta\,t}(\vu_h^n- \vu_h^{n-1}) \|^2_{H^{-1}(\Om)} \le 3C^2 \Big(&\alpha^2\big(\frac{h}{\Delta\,t}\big)^2\|\vu_h^n- \vu_h^{n-1}\|^2_{L^2(\Om)}
 +(\alpha \cP)^2 \|\vv_h^n\|^2_{L^2(\Om)} \\
 &+ \big(\frac{\gamma_2}{E_1}\big)^2\cD_2  d_2(\vu_h^n)\Big),
\end{split}
\end{equation}
where $C$ is the constant of \eqref{eq:PihC}.
\end{prop}

\begin{proof}
By definition, we have
$$
\|\frac{1}{\Delta\,t} (\vu_h^n- \vu_h^{n-1}) \|_{H^{-1}(\Om)} = \sup_{\vvarphi \in H^1_0(\Om)^d}
\frac{1}{\|\nabla\,\vvarphi\|_{L^2(\Om)}}\int_\Om \frac{1}{\Delta\,t}(\vu_h^n- \vu_h^{n-1}) \cdot \vvarphi.
$$
Now,
$$
\frac{1}{\Delta\,t}(\vu_h^n- \vu_h^{n-1}) \cdot \vvarphi = \frac{1}{\Delta\,t}(\vu_h^n- \vu_h^{n-1}) \cdot ( \vvarphi - \Pi_h \vvarphi) +\frac{1}{\Delta\,t}(\vu_h^n- \vu_h^{n-1}) \cdot \Pi_h\vvarphi,
$$
where $\Pi_h \vvarphi \in \vU_h$  interpolates $ \vvarphi$ on $\vU_h$. Then
$$
\int_\Om\big|\frac{1}{\Delta\,t}(\vu_h^n- \vu_h^{n-1}) \cdot \vvarphi \big| \le  C \frac{h}{\Delta\,t}\|\vu_h^n- \vu_h^{n-1}\|_{L^2(\Om)} \|\nabla\,\vvarphi\|_{L^2(\Om)}  +\int_\Om\big|\frac{1}{\Delta\,t}(\vu_h^n- \vu_h^{n-1}) \cdot \Pi_h\vvarphi\big|.
$$
Taking $\vz_h=\Pi_h\vvarphi$ in \eqref{eq:displachn} and invoking the Poincar\'e inequality \eqref{eq:Poincare} along with the stability of $\Pi_h$ yields
\begin{align*}
&\alpha \int_\Om\big|\frac{1}{\Delta\,t}(\vu_h^n- \vu_h^{n-1}) \cdot \Pi_h\vvarphi\big|\\
& \le C \Big(\alpha \cP \|\vv_h^n\|_{L^2(\Om)}  + \frac{\gamma_2}{E_1}\big(\|\veps(\vu_h^n)\|_{L^2(\Om)} + \gamma_6 \gamma_8 |E_2| \|\di\,\vu_h^n\| _{L^2(\Om)}\big)\Big) \|\nabla\,\vvarphi\|_{L^2(\Om)}.
\end{align*}
Hence, we obtain
\begin{equation}
\label{eq:uhn-1}
\begin{split}
\alpha \|\frac{1}{\Delta\,t}(\vu_h^n&- \vu_h^{n-1}) \|_{H^{-1}(\Om)} \le C \Big(\alpha\frac{h}{\Delta\,t}\|\vu_h^n- \vu_h^{n-1}\|_{L^2(\Om)}\\
& +\alpha \cP \|\vv_h^n\|_{L^2(\Om)}  + \frac{\gamma_2}{E_1}\big(\|\veps(\vu_h^n)\|_{L^2(\Om)} + \gamma_6 \gamma_8 |E_2| \|\di\,\vu_h^n\| _{L^2(\Om)}\big)\Big),
\end{split}
\end{equation}
and
\begin{equation}
\label{eq:uhn2-1}
\begin{split}
\alpha^2 \|\frac{1}{\Delta\,t}(\vu_h^n&- \vu_h^{n-1}) \|^2_{H^{-1}(\Om)} \le 3C^2 \Big(\alpha^2\big(\frac{h}{\Delta\,t}\big)^2\|\vu_h^n- \vu_h^{n-1}\|^2_{L^2(\Om)}\\
& +(\alpha \cP)^2 \|\vv_h^n\|^2_{L^2(\Om)}  + \big(\frac{\gamma_2}{E_1}\big)^2\big(\|\veps(\vu_h^n)\|_{L^2(\Om)} + \gamma_6 \gamma_8 |E_2| \|\di\,\vu_h^n\| _{L^2(\Om)}\big)^2\Big).
\end{split}
\end{equation}
To prove \eqref{eq:uhn2-1.1}, it remains to note that
$$
\big(\|\veps(\vu_h^n)\|_{L^2(\Om)} + \gamma_6 \gamma_8 |E_2 \|\di\,\vu_h^n\| _{L^2(\Om)}\big)^2 
\le  \cD_2 d_2(\vu_h^n),
$$
which follows from a weighted Cauchy-Schwarz inequality.
\end{proof}

By substituting \eqref{eq:uhn2-1.1} into \eqref{eq:vhn1}, multiplying by $\Delta\,t$ and summing over $n$ we obtain
\begin{equation}
\label{eq:vhn2}
\begin{split}
&\varrho\big(\|\vv_h^n\|^2_{L^2(\Om)}  + \sum_{i=1}^n\|\vv_h^i- \vv_h^{i-1} \|^2_{L^2(\Om)}\big)
 +  \nu \sum_{i=1}^n \Delta\,t \| \nabla\, \vv_h^i\|^2_{L^2(\Om)}\\
 &+ 2\alpha \sum_{i=1}^n \Delta\,t\|\vv_h^i\|^2_{L^2(\Om)}
 \le \varrho\|\vv_h^0\|^2_{L^2(\Om)} 
 +\frac{3}{\nu}C^2 \Big[\alpha^2\big(\sum_{i=1}^n\frac{h^2}{\Delta\,t}\|\vu_h^i- \vu_h^{i-1}\|^2_{L^2(\Om)}
 \\
 &+\cP^2 \sum_{i=1}^n {\Delta\,t}\|\vv_h^i\|^2_{L^2(\Om)}\big)
 + \big(\frac{\gamma_2}{E_1}\big)^2\cD_2\sum_{i=1}^n (\Delta \,t) d_2(\vu_h^i)\Big].
\end{split}
\end{equation}
From here, to derive an estimate for $\vv_h^n$, we must estimate the terms involving $\vu_h^i$, $0\le i\le n$, in terms of $\vv_h^n$.  For this, multiplying \eqref{eq:uhn1} by $ \Delta\,t$ and summing over $n$, we obtain
\begin{equation}
\label{eq:uhn2}
\begin{split}
\alpha\big(\|\vu_h^n\|^2_{L^2(\Om)}&  + \sum_{i=1}^n\|\vu_h^i- \vu_h^{i-1} \|^2_{L^2(\Om)}\big)
+ \frac{\gamma_1}{E_1} \sum_{i=1}^n (\Delta \,t) d_2(\vu_h^i)\\
& \le \alpha\|\vu_h^0\|^2_{L^2(\Om)} +  \frac{E_1}{\gamma_1}(\alpha \cP \cK)^2\sum_{i=1}^n \Delta\,t\|\vv_h^i\|^2_{L^2(\Om)}.
\end{split}
\end{equation}
This implies on the one hand
\begin{equation}
\label{eq:uhn3}
\sum_{i=1}^n (\Delta \,t) d_2(\vu_h^i)\\
 \le \frac{E_1}{\gamma_1}\Big(\alpha\|\vu_h^0\|^2_{L^2(\Om)} +  \frac{E_1}{\gamma_1}(\alpha \cP \cK)^2\sum_{i=1}^n \Delta\,t\|\vv_h^i\|^2_{L^2(\Om)}\Big),
\end{equation}
and on the other hand
\begin{equation}
\label{eq:uhn4}
\sum_{i=1}^n\|\vu_h^i- \vu_h^{i-1} \|^2_{L^2(\Om)} \le \|\vu_h^0\|^2_{L^2(\Om)} + \frac{E_1}{\gamma_1} \alpha (\cP \cK)^2 \sum_{i=1}^n \Delta\,t\|\vv_h^i\|^2_{L^2(\Om)}.
\end{equation}
With \eqref{eq:uhn3} and \eqref{eq:uhn4}, \eqref{eq:vhn2} becomes
\begin{equation}
\label{eq:vhn3}
\begin{split}
&\varrho\big(\|\vv_h^n\|^2_{L^2(\Om)}  + \sum_{i=1}^n\|\vv_h^i- \vv_h^{i-1} \|^2_{L^2(\Om)}\big)
 +  \nu \sum_{i=1}^n \Delta\,t \| \nabla\, \vv_h^i\|^2_{L^2(\Om)}  + 2\alpha \sum_{i=1}^n \Delta\,t\|\vv_h^i\|^2_{L^2(\Om)}\\
 &\le \varrho\|\vv_h^0\|^2_{L^2(\Om)} 
 +\frac{3C^2}{\nu} \Big[\alpha^2\frac{h^2}{\Delta \,t}\big(
 \|\vu_h^0\|^2_{L^2(\Om)} + \frac{E_1}{\gamma_1} \alpha (\cP \cK)^2 \sum_{i=1}^n \Delta\,t\|\vv_h^i\|^2_{L^2(\Om)}\big)\\
 &+ (\alpha \cP)^2 \sum_{i=1}^n \Delta \,t \|\vv_h^i\|^2_{L^2(\Om)}
 + \frac{\gamma_2^2}{E_1 \gamma_1}\cD_2\big(\alpha \|\vu_h^0\|^2_{L^2(\Om)}
 + \frac{E_1}{\gamma_1}(\alpha \cP \cK)^2
 \sum_{i=1}^n \Delta \,t\|\vv_h^i\|^2_{L^2(\Om)} \big) \Big]. 
\end{split}
\end{equation}
To be concise, we introduce two more constants
\begin{equation}
\label{eq:cDi}
\cD_3 :=\frac{(\alpha h)^2}{\Delta \,t} + \frac{\alpha}{E_1}\frac{\gamma_2^2}{\gamma_1}\cD_2, \quad 
    \cD_4 := (\alpha \cP)^2\big(1 + \frac{E_1 \cK^2}{\alpha \gamma_1} \cD_3\big).
    \end{equation}
Under the mild restriction
\begin{equation}
    \label{eq:h2det}
 h^2 \le \Delta \,t,   
\end{equation}
these constants are independent of $h$ and $\Delta\,t$.
Then by collecting terms and dividing by $\varrho$, \eqref{eq:vhn3} becomes
\begin{equation}
\label{eq:vhn4}
\begin{split}
&\|\vv_h^n\|^2_{L^2(\Om)}  + \sum_{i=1}^n\|\vv_h^i- \vv_h^{i-1} \|^2_{L^2(\Om)}
 +  \frac{\nu}{\varrho} \sum_{i=1}^n \Delta\,t \| \nabla\, \vv_h^i\|^2_{L^2(\Om)}  + 2\frac{\alpha}{\varrho} \sum_{i=1}^n \Delta\,t\|\vv_h^i\|^2_{L^2(\Om)}\\
 &\le \|\vv_h^0\|^2_{L^2(\Om)} + \frac{3}{\nu\varrho}C^2\big(\cD_3 \|\vu_h^0\|^2_{L^2(\Om)} + \cD_4 \sum_{i=1}^n \Delta\,t\|\vv_h^i\|^2_{L^2(\Om)}\big).
\end{split}
\end{equation}
To apply the discrete Gronwall's lemma, we write
$$
\Delta\,t \|\vv_h^n\|^2_{L^2(\Om)} \le 2 \Delta\,t \|\vv_h^{n-1}\|^2_{L^2(\Om)} + 2 \Delta\,t \|\vv_h^n-\vv_h^{n-1}\|^2_{L^2(\Om)},
$$
and suppose $\Delta\,t$ is small enough so that
\begin{equation}
    \label{eq:delt.t}
    6   \frac{C^2 \cD_4}{\nu \varrho} \Delta\,t\le 1.
\end{equation}
Then
\begin{align*}
&\|\vv_h^n\|^2_{L^2(\Om)}  +  \sum_{i=1}^{n-1}\|\vv_h^i- \vv_h^{i-1} \|^2_{L^2(\Om)}
 +  \frac{\nu}{\varrho} \sum_{i=1}^n \Delta\,t \| \nabla\, \vv_h^i\|^2_{L^2(\Om)}  + 2\frac{\alpha}{\varrho} \sum_{i=1}^n \Delta\,t\|\vv_h^i\|^2_{L^2(\Om)}\\
 &\le \|\vv_h^0\|^2_{L^2(\Om)} + \frac{3}{\nu\varrho}C^2\Big(\cD_3 \|\vu_h^0\|^2_{L^2(\Om)} + \cD_4 \big( 3  \Delta\,t  \|\vv_h^{n-1}\|^2_{L^2(\Om)} +  \sum_{i=1}^{n-2} \Delta\,t\|\vv_h^i\|^2_{L^2(\Om)}\big) \Big).
\end{align*}

By applying the discrete Gronwall's lemma, we obtain an a priori estimate for $\vv_h^n$.

\begin{thm}
\label{thm:bddvh}
Let $\bar \lambda>0$ be as in \eqref{e:small_trace} and assume that $|\lambda_1|, |\lambda_2| \leq \bar \lambda$ are sufficiently small so that \eqref{e:gam12}, \eqref{eq:gam56.1},  and \eqref{eq:gam78.1} hold. Assume that the time-step $\Delta\, t$ satisfies the restriction \eqref{eq:delt.t}. 
Then the sequence $\vv_h^n$ satisfies for all $1\le n \le N$,
\begin{equation}
\label{eq:vhn5}
\begin{split}
\|\vv_h^n\|^2_{L^2(\Om)}  &+ \sum_{i=1}^{n-1}\|\vv_h^i- \vv_h^{i-1} \|^2_{L^2(\Om)}
 +  \frac{\nu}{\varrho} \sum_{i=1}^n \Delta\,t \| \nabla\, \vv_h^i\|^2_{L^2(\Om)}\\
 &  + 2\frac{\alpha}{\varrho} \sum_{i=1}^n \Delta\,t\|\vv_h^i\|^2_{L^2(\Om)}
 \le C_1 {\rm exp}(C_2t_n),
\end{split}
\end{equation}
where
$$
C_1:= \|\vv_h^0\|^2_{L^2(\Om)} + \frac{3C^2\cD_3}{\nu\varrho}  \|\vu_h^0\|^2_{L^2(\Om)} , \quad C_2 := \frac{9 C^2\cD_4}{\nu\varrho},
$$
and the constants $\cD_i$ are given by \eqref{eq:cD1} and \eqref{eq:cDi}. If \eqref{eq:h2det} holds, then these constants are independent of $h$ and $\Delta\, t$.
\end{thm}

By substituting this result into \eqref{eq:uhn2}, we immediately derive an estimate for $\vu_h^n$.

\begin{cor}
\label{cor:bdduhn}
With the notation and assumptions of Theorem \ref{thm:bddvh}, the sequence $\vu_h^n$ satisfies for $1\le n \le N$,
\begin{equation}
\label{eq:uhn5}
\begin{split}
\|\vu_h^n\|^2_{L^2(\Om)}  + &\sum_{i=1}^n\|\vu_h^i- \vu_h^{i-1} \|^2_{L^2(\Om)}
+ \frac{\gamma_1}{\alpha E_1} \sum_{i=1}^n (\Delta\,t)d_2(\vu_h^i)\\
& \le \|\vu_h^0\|^2_{L^2(\Om)} + \frac{E_1}{\gamma_1} \frac{(\cP \cK)^2}{2} \varrho C_1 {\rm exp}(C_2t_n). 
\end{split}
\end{equation}
\end{cor}

We end this section with a pressure estimate. Note that an estimate cannot be established in $L^2(0,T;L^2_0(\Om))$, but in $W^{-1, \infty}_0(0,T;L^2_0(\Om))$, the dual space of $W^{1,1}(0,T;L^2_0(\Om))$. For its derivation, it is convenient to introduce the step function
\begin{equation}
    \label{eq:funcp_h}
 \forall\, 1 \le i \le N,\quad    p_h(t)|_{]t_{i-1},t_i]} = p_h^i. 
\end{equation}
Since $p_h^i$ belongs to $L^2_0(\Om)$ for all $i$, the function $p_h(t)$ belongs to $L^2_0(\Om)$ for all $t$, $0<t \le T$.
We then introduce the sequence of sums
\begin{equation}\label{e:Phn}
P_h^n = \sum_{i=1}^n \Delta\,t p_h^i= \int_0^{t_n} p_h(t).
\end{equation}
Then $P_h^0 = 0$, $P_h^n$ belongs to $Q_h \subset L^2_0(\Om)$, and
\begin{equation}
    \label{eq:PhPhn-1}
\forall\, 1 \le i \le n,\quad    p_h^i = \frac{1}{\Delta \,t} \big(P_h^i-P_h^{i-1}\big).
\end{equation}
In virtue of \eqref{eq:fluidhn}, $P_h^n$ solves for all $1 \le n  \le N$,
\begin{equation}
     \label{eq:Phn}
     \begin{split}
     \int_\Om P_h^n\di\,\vw_h =& 
   \varrho\int_\Om (\vv_h^n- \vv_h^{0})\cdot \vw_h
 +   \nu \sum_{i=1}^n\Delta\,t \int_\Om \nabla\, \vv_h^i : \nabla\,
\vw_h\\
& + \alpha \sum_{i=1}^n\Delta\,t\int_\Om \vv_h^i \cdot \vw_h 
- \alpha\int_\Om (\vu_h^n- \vu_h^0)\cdot\vw_h.
\end{split}
 \end{equation}
The discrete inf-sup condition easily implies that
\begin{align*}
     \|P_h^n\|_{L^2(\Om)} \le& \frac{1}{\beta^*}\Big( 
\varrho \cP \| \vv_h^n - \vv_h^{0}\|_{L^2(\Om)} +
\nu \sqrt{T} \big( \sum_{i=1}^n \Delta\,t\|\nabla \,\vv_h^i\|^2_{L^2(\Om)}\big)^\frac{1}{2}\\
& + \alpha \cP \sqrt{T} \big( \sum_{i=1}^n \Delta\,t\|\vv_h^i\|^2_{L^2(\Om)}\big)^\frac{1}{2} + 
\alpha \cP \| \vu_h^n - \vu_h^{0}\|_{L^2(\Om)}\Big),
\end{align*}
where $\beta^*$ is the constant of the discrete inf-sup condition. Then, the estimates of Theorem \ref{thm:bddvh} and Corollary \ref{cor:bdduhn} and a weighted Cauchy-Schwarz inequality yield the upper bound, valid for all $1 \le n \le N$
\begin{equation}
     \label{eq:|Ph|}
    \|P_h^n\|_{L^2(\Om)} \le C_3 + C_4 \big( C_1 \mbox{exp} (C_2t_n)\big)^\frac{1}{2},
    \end{equation}
where 
\begin{equation}
    \label{C_3}
C_3 := \frac{1}{\beta^*} \cP\big(\varrho \|\vv_h^{0}\|_{L^2(\Om)} + 2\alpha  \|\vu_h^{0}\|_{L^2(\Om)}\big),
 \end{equation}
 \begin{equation}
    \label{C_4}
C_4 := \frac{\cP}{\beta^*}\Big( \alpha \cP \cK \big( \frac{E_1\varrho }{2 \gamma_1}\big)^\frac{1}{2} + \varrho\big(1 + \frac{T}{\varrho}( \frac{\nu}{\cP^2} + \frac{\alpha}{2})\big)^\frac{1}{2} \Big).
\end{equation}

The next proposition gives a pressure estimate.

\begin{prop}
\label{pro:bddpres}
Under the assumptions of Theorem \ref{thm:bddvh}, $p_h$ given in \eqref{eq:funcp_h} satisfies $p_h\in W^{-1, \infty}(0,T;L^2_0(\Om))$ with
\begin{equation}
\label{eq:bddp_h(t)}
\|p_h\|_{W^{-1, \infty}(0,T;L^2(\Om))} \le 3\Big( C_3 + C_4 \big( C_1 {\rm exp} C_2T)\big)^\frac{1}{2}\Big).
\end{equation}
\end{prop}

\begin{proof}
Let $\varphi$ be a function in $W^{1, 1}_0(0,T;L^2_0(\Om))$; then
\begin{align*}
\int_0^{T} \int_\Om p_h(t) \varphi(t) = \sum_{i=1}^N \int_{t_{i-1}}^{t_i} \int_\Om p_h^i \varphi(t) &= \sum_{i=1}^N \int_{t_{i-1}}^{t_i} \int_\Om p_h^i (\varphi(t) - \varphi(t_i))\\
&+ \sum_{i=1}^N \int_{t_{i-1}}^{t_i} \int_\Om p_h^i \varphi(t_i).
\end{align*}
Let $m_i(\varphi)$ denote the average of $\varphi$ over the interval $]t_{i-1}, t_i[$,
\begin{equation}
\label{eq:averag}
m_i(\varphi) = \frac{1}{\Delta \,t} \int_{t_{i-1}}^{t_i} \varphi(t).
\end{equation}
Then
\begin{align*}
\sum_{i=1}^N &\int_{t_{i-1}}^{t_i} \int_\Om p_h^i (\varphi(t) - \varphi(t_i)) = \sum_{i=1}^N \int_\Om  \Delta \,t p_h^i \big(m_i(\varphi) - \varphi(t_i)\big)\\
&= \sum_{i=1}^N \int_\Om \big(P_h^i - P_h^{i-1}\big)\big(m_i(\varphi) - \varphi(t_i)\big)\\
& = \sum_{i=1}^{N-1}\int_\Om P_h^i\big( m_i(\varphi) - \varphi(t_i) - (m_{i+1}(\varphi) -\varphi(t_{i+1})\big) + \int_\Om P_h^N \big( m_N(\varphi) - \varphi(T)\big).
\end{align*}
Similarly,
$$
\sum_{i=1}^N \int_{t_{i-1}}^{t_i} \int_\Om p_h^i \varphi(t_i) = \sum_{i=1}^N \int_\Om\big(P_h^i - P_h^{i-1}\big)\varphi(t_i)
 = \sum_{i=1}^{N-1}\int_\Om P_h^i\big( \varphi(t_i) -\varphi(t_{i+1})\big), 
$$
where this time we have also used that $\varphi(T) =0 $. 
Thus,
\begin{equation}
    \label{eq:ph.phi}
    \begin{split}
\int_0^{T} \int_\Om p_h(t) \varphi(t) &=  \sum_{i=1}^{N-1}  \int_\Om  P_h^i\big( m_i(\varphi) - m_{i+1}(\varphi) - (\varphi(t_i) -\varphi(t_{i+1})\big)\\
&+ \int_\Om P_h^N \big( m_N(\varphi) - \varphi(T)\big)+ \sum_{i=1}^{N-1}\int_\Om P_h^i\big( \varphi(t_i) -\varphi(t_{i+1})\big).
 \end{split}
\end{equation}
But
\begin{equation}
    \label{eq:mif-f}
m_i(\varphi) - \varphi(t_i) = - \frac{1}{\Delta\,t} \int_{t_{i-1}}^{t_i} \int_{t}^{t_i} \pdd_t\varphi(s)\, \nd s\, \nd t.
\end{equation}
Hence
\begin{equation}
    \label{eq:mi-phi}
 \big|\int_\Om   P_h^i\big( m_i(\varphi)- \varphi(t_i)\big)\big| \le \| P_h^i\|_{L^2(\Om)} \|\pdd_t\varphi \|_{L^1(t_{i-1},t_i; L^2(\Om))}.
\end{equation}
By applying \eqref{eq:mi-phi} in \eqref{eq:ph.phi}, we arrive at
\begin{align*}
&\big| \int_0^{T} \int_\Om p_h(t) \varphi(t) \big| \le \sum_{i=1}^{N-1} \| P_h^i\|_{L^2(\Om)}\Big( \|\pdd_t\varphi \|_{L^1(t_{i-1},t_i; {L^2(\Om)})} + 
\|\pdd_t\varphi \|_{L^1(t_{i},t_{i+1}; L^2_0(\Om))}\Big)\\
&+ \|P_h^N\|_{{L^2(\Om)}}\|\pdd_t\varphi \|_{L^1(t_{N-1},t_{N}; {L^2(\Om)})}
 + \sum_{i=1}^{N-1} \| P_h^i\|_{{L^2(\Om)}} \|\pdd_t\varphi \|_{L^1(t_{i},t_{i+1}; {L^2(\Om)})}\\
& \le 2 \sum_{i=1}^{N-1} \| P_h^i\|_{{L^2(\Om)}} \|\pdd_t\varphi \|_{L^1(t_{i},t_{i+1}; {L^2(\Om)})} + \sum_{i=1}^{N} \| P_h^i\|_{{L^2(\Om)}} \|\pdd_t\varphi \|_{L^1(t_{i-1},t_i; {L^2(\Om)})}.
\end{align*}
Then \eqref{eq:bddp_h(t)} follows by substituting \eqref{eq:|Ph|} into this inequality.
    \end{proof}

%--------------------------------
\section{A priori error estimates}
 \label{sec:Errorbouds}

This section is written with the assumption that Problem \eqref{variational} has a solution that satisfies \eqref{additional_reg}, {\em but not necessarily small}. Also, although the theory is developed for an arbitrary space discretization, for the sake of simplicity, the error estimates are written for a space approximation of order one. 
Our aim is to derive error estimates for the fluid velocity, the solid displacement, and the fluid pressure as stated in the next result.  For simplicity, we assume that the finite element spaces are first order, i.e.,
\begin{equation}\label{e:ass_error}
  \inf_{\vz_h \in \vU_h} \|\vz - \vz_h\|_{H^1(\Omega)}
  + \inf_{\vw_h \in \vX_h} \| \vw-\vw_h\|_{H^1(\Omega)} + 
 \inf_{q_h \in Q_h} \| q - q_h\|_{L^2(\Omega)}
 \in\mathcal{O} (h),
\end{equation}
for all $(\vw,\vz,q)\in H_0^1(\Omega)^d\times H_0^1(\Omega)^d \times L_0^2(\Omega)$.
 
\begin{thm}
\label{thm:velerr}
Assume that the exact solution satisfies the smoothness assumption \eqref{additional_reg}. 
Let $\bar \lambda>0$ be as in \eqref{e:small_trace} and assume that $|\lambda_1|, |\lambda_2| \leq \bar \lambda$ are sufficiently small so that \eqref{e:gam12}, \eqref{eq:gam56.1},  and \eqref{eq:gam78.1} hold. Furthermore, assume that discretization parameters satisfy $h^2 \leq c \Delta\,t$ for an absolute positive constant $c$ and that  $\Delta \,t$, $|\lambda_1|$, and $|\lambda_2|$ are sufficiently small.
Then, there are positive constants $c_1$, $c_2$ and $c_3$ independent of $\Delta \, t$ and $h$ such that for all $1\leq n \leq N$
\begin{equation}
\label{eq:3errvfl}
\|\vv(t_n)-\vv_h^n\|^2_{L^2(\Om)}   + \Delta \,t  \sum_{i=1}^n \|\vv(t_i)-\vv_h^i\|^2_{H^1(\Om)} \le c_1\big(h^2 + (\Delta \,t)^2\big) e^{c_2 t_n},
\end{equation}
\begin{equation}
\label{eq:2err.disp}
\begin{split}
 \|\vu(t_n)-\vu_h^n\|^2_{L^2(\Om)} & +  \Delta\,t \sum_{i=1}^n \left(  \|\veps(\vu(t_i)-\vu_h^i)\|^2_{L^2(\Om)} +  \|\di\,(\vu(t_i)-\vu_h^i)\|^2_{L^2(\Om)}\right)
\\
&\le c_1 \big(h^2 + (\Delta \,t)^2\big) \big(
e^{c_2 t_n} + c_3\big),
\end{split}
\end{equation}
and
\begin{equation}
    \label{eq:}
\|p-p_h\|_{W^{-1,\infty}(0,T;{L^2(\Om)})} \le  
c_1 \left(h + (\Delta \,t)\right)\left( e^{c_2 T} +  c_3\right),
\end{equation}
where $p_h(t) = \sum_{i=1}^N p_h^i \chi_{]t_{i-1},t_i]}$ and $\chi_A$ is the characteristic function of $A$.
\end{thm}

We first note that in view of the assumed regularity \eqref{additional_reg} on the solution, it suffices to prove the above estimate for the discrete errors $\ve_v^n:= \Pi_h \vv(t_n)-\vv_h^n$, $\ve_u^n:= \Pi_h \vu(t_n)-\vu_h^n$, and $e_p^n:= \Pi_h p(t_n)-p_h$.
To achieved this, we proceed in several steps. In Section~\ref{subsec:error.eq}, we obtain consistency relations satisfied by the discrete errors. The resulting consistency terms are then estimated in \ref{subsec:bddconsist}.
We then proceed in Sections~\ref{subsec:1stbdddispl} and~\ref{subsec:1srbdd.dtu} by deriving estimates for the displacement discrete error $\ve_u^n$ and its discrete time derivative involving the fluid velocity. In Section~\ref{subsec:err.bddvel.fl}, we obtain estimates for the fluid velocity discrete error $\ve_v^n$ involving the displacement and deduce estimates \eqref{eq:3errvfl} and \eqref{eq:2err.disp}.
For the pressure error, a consistency relation for the average pressure is first obtained in Section~\ref{subsec:err.press} and the consistency terms are estimated in Section~\ref{subsec:bddcons.press}. Estimate \eqref{eq:} for the fluid pressure error is then derived in Section~\ref{subsec:err.aver.press}.

In this section we do not track the precise dependency on the physical constants. Instead, we denote by $c$, $c_\lambda$ generic positive constant independent of $h$ and $n$, with $c_\lambda$ monotone decreasing in $|\lambda_1|$ and $|\lambda_2|$ and 
$$
\lim_{|\lambda_1|+|\lambda_2|\to 0} c_\lambda = 0.
$$
We also use the notation $A \lesssim B$ to indicate $A \leq c B$. The value of $c$ and $c_{\lambda}$ might change from one line to another.

\subsection{Consistency}
\label{subsec:error.eq}

Let us start by deriving a consistency relation for the displacement. Assuming that all terms below are continuous in time, the exact solution at time $t_n$ satisfies
\begin{align*}
&\alpha \langle \pdd_t \vu(t_n), \vz_h\rangle 
 + \frac{1}{E_1}\int_\Om \frac{\veps( \vu(t_n)): \veps( \vz_h)}{1+ \lambda_1T_{\delta} \di\, \vu(t_n)}  \\
& + \frac{|E_2|}{E_1} \int_\Om \frac{ 1+ \lambda_2T_{\delta} \di\, \vu(t_n)}{( 1+ \lambda_1T_{\delta} \di\, \vu(t_n))}
\frac{(\di\,\vu(t_n)) (\di\,\vz_h)}{F(T_{\delta} \di\, \vu(t_n))}
   = \alpha \int_\Om  \vv(t_n) \cdot \vz_h    
\end{align*}
This can also be written as
\begin{align*}
&\frac{\alpha}{\Delta\,t} \int_\Om \big(\vu(t_n) - \vu(t_{n-1})\big)\cdot \vz_h 
 + \frac{1}{E_1}\int_\Om \frac{\veps( \vu(t_n)): \veps( \vz_h)}{1+ \lambda_1T_{\delta} \di\, \vu(t_{n-1})}  \\
& + \frac{|E_2|}{E_1} \int_\Om \frac{ 1+ \lambda_2T_{\delta} \di\, \vu(t_{n-1})}{( 1+ \lambda_1T_{\delta} \di\, \vu(t_{n-1}))}
\frac{(\di\,\vu(t_n)) (\di\,\vz_h)}{F(T_{\delta} \di\, \vu(t_{n-1}))}
  \\
& 
 = \alpha \int_\Om  \vv(t_n) \cdot \vz_h + R_{1,u}^n(\vz_h)    
\end{align*}
with the first consistency term 
\begin{equation}
\label{eq:R1u}
\begin{split}
&R_{1,u}^n(\vz_h) := \alpha \langle \frac{1}{\Delta\,t}\big(\vu(t_n) - \vu(t_{n-1})\big) -  \pdd_t \vu(t_n), \vz_h\rangle\\
& +  \frac{1}{E_1}\int_\Om \veps( \vu(t_n)): \veps( \vz_h)\big( \frac{1}{1+ \lambda_1T_{\delta} \di\, \vu(t_{n-1})} - \frac{1}{1+ \lambda_1T_{\delta} \di\, \vu(t_{n})}\big)\\
&+  \frac{|E_2|}{E_1} \int_\Om (\di\,\vu(t_n)) (\di\,\vz_h) \Big(
\frac{ 1+ \lambda_2T_{\delta} \di\, \vu(t_{n-1})}{( 1+ \lambda_1T_{\delta} \di\, \vu(t_{n-1}))}
\frac{1}{F(T_{\delta}}\di\, \vu(t_{n-1}))\\
&- \frac{ 1+ \lambda_2T_{\delta} \di\, \vu(t_{n})}{( 1+ \lambda_1T_{\delta} \di\, \vu(t_{n}))}
\frac{1}{F( T_{\delta} \di\, \vu(t_{n}))}\Big).
\end{split}
\end{equation}
By inserting interpolants, we obtain a suitable consistency equality for the displacement,
\begin{equation}
\label{eq:consist.equ}
\begin{split}
&\frac{\alpha}{\Delta\,t} \int_\Om \big(\Pi_h(\vu(t_n) - \vu(t_{n-1}))\big)\cdot \vz_h 
 + \frac{1}{E_1}\int_\Om \frac{\veps( \Pi_h\vu(t_n))): \veps( \vz_h)}{1+ \lambda_1T_{\delta} \di\, \Pi_h\vu(t_{n-1})}  \\
& + \frac{|E_2|}{E_1} \int_\Om \frac{ 1+ \lambda_2T_{\delta} \di\, \Pi_h \vu(t_{n-1}))}{( 1+ \lambda_1T_{\delta} \di\, \Pi_h\vu(t_{n-1}))}
\frac{(\di\,\Pi_h\vu(t_n)) (\di\,\vz_h)}{F(T_{\delta} \di\, \Pi_h\vu(t_{n-1}))}
  \\
& 
 = \alpha \int_\Om  \Pi_h\vv(t_n) \cdot \vz_h + R_{1,u}^n(\vz_h) + R_{2,u}^n(\vz_h),
\end{split}
 \end{equation}
with the second consistency term
\begin{equation}
\label{eq:R2u}
\begin{split}
&R_{2,u}^n(\vz_h) := \frac{\alpha}{\Delta\,t} \int_\Om \big(\Pi_h(\vu(t_n)-\vu(t_{n-1}))-(\vu(t_n) - \vu(t_{n-1}))\big)\cdot  \vz_h\\\
& +  \frac{1}{E_1}\int_\Om \Big(\frac{\veps( \Pi_h\vu(t_n)): \veps( \vz_h)} {1+ \lambda_1T_{\delta} \di\, \Pi_h\vu(t_{n-1})}- \frac{\veps( \vu(t_n)): \veps( \vz_h)} {1+ \lambda_1T_{\delta} \di\, \vu(t_{n-1})}\Big)\\
&+  \frac{|E_2|}{E_1}\Big[\int_\Om 
\frac{ 1+ \lambda_2T_{\delta} \di\, \Pi_h \vu(t_{n-1})}{( 1+ \lambda_1T_{\delta} \di\, \Pi_h\vu(t_{n-1}))}
\frac{(\di\,\Pi_h\vu(t_n)) (\di\,\vz_h)}{F(T_{\delta} \di\, \Pi_h\vu(t_{n-1}))}\\
&-\int_\Om  
\frac{ 1+ \lambda_2T_{\delta} \di\, \vu(t_{n-1})}{( 1+ \lambda_1T_{\delta} \di\, \vu(t_{n-1}))}
\frac{(\di\,\vu(t_n)) (\di\,\vz_h)}{F(T_{\delta} \di\, \vu(t_{n-1}))}\Big]\\
& + \alpha \int_\Om  \big(\vv(t_n)-\Pi_h\vv(t_n)\big) \cdot \vz_h.
\end{split}
\end{equation}

For the fluid variables, assuming again that all terms below are continuous in time, the exact solution at time $t_n$ satisfies for all $\vw_h \in \vV_h$,
\begin{align*}
\varrho\langle \pdd_t \vv(t_n), \vw_h \rangle -
  \int_\Om p(t_n)\di\,\vw_h + \nu \int_\Om \nabla\, \vv(t_n) : \nabla\,
\vw_h & + \alpha \int_\Om \vv(t_n) \cdot \vw_h \\
&
= \alpha \langle \pdd_t\vu(t_n),\vw_h \rangle,
\end{align*}
that we first write for any $q_h$ in $Q_h$ as
\begin{align*}
  \frac{\varrho}{\Delta\,t}\int_\Om (\vv(t_n)- \vv(t_{n-1}))\cdot \vw_h
+ \nu &\int_\Om \nabla\, \vv(t_n) : \nabla\, 
\vw_h    + \alpha \int_\Om \vv(t_n) \cdot \vw_h \\
&= \frac{\alpha}{\Delta\,t}\int_\Om (\vu(t_n)- \vu(t_{n-1}))\cdot\vw_h + R_{1,v}^n(\vw_h,q_h),
\end{align*}
where the first consistency term is
\begin{equation}
\label{eq:R1v}
\begin{split}
R_{1,v}^n(\vw_h,q_h) &:=  \int_\Om \big(p(t_n)-q_h \big)\di\,\vw_h + \langle \frac{\varrho}{\Delta\,t}(\vv(t_n)- \vv(t_{n-1})) - \pdd_t \vv(t_n), \vw_h\rangle  \\
& -\alpha\langle \frac{1}{\Delta\,t}(\vu(t_n)- \vu(t_{n-1}))- \pdd_t \vu(t_n), \vw_h\rangle.
\end{split}
\end{equation}
Inserting interpolants again, we obtain a suitable consistency equality for the fluid velocity,
\begin{equation}
\label{eq:consist.eqv}
\begin{split}
\frac{\varrho}{\Delta\,t}&\int_\Om (\Pi_h(\vv(t_n)-\vv(t_{n-1}))\cdot \vw_h
+ \nu \int_\Om \nabla\, \Pi_h\vv(t_n) : \nabla\, 
\vw_h    + \alpha \int_\Om \Pi_h\vv(t_n) \cdot \vw_h \\
&= \frac{\alpha}{\Delta\,t}\int_\Om (\Pi_h(\vu(t_n) - \vu(t_{n-1}))\cdot\vw_h + R_{1,v}^n(\vw_h,q_h) + R_{2,v}^n(\vw_h),
\end{split}
\end{equation}
with the second consistency term
\begin{equation}
\label{eq:R2v}
\begin{split}
R_{2,v}^n(\vw_h) &:= \frac{\varrho}{\Delta\,t}\int_\Om \big(\Pi_h(\vv(t_n) -\vv(t_{n-1})) - (\vv(t_n)- \vv(t_{n-1})) \cdot \vw_h\\
& + \nu \int_\Om \nabla( \Pi_h\vv(t_n)-\vv(t_n))  : \nabla\, 
\vw_h  + \alpha \int_\Om (\Pi_h\vv(t_n) - \vv(t_n)) \cdot \vw_h\\
& +\frac{\alpha}{\Delta\,t}\int_\Om (\vu(t_n)- \vu(t_{n-1}) - (\Pi_h(\vu(t_n)- \vu(t_{n-1})) )\cdot\vw_h. 
\end{split}
\end{equation}

The following preliminary error equality for the displacement stems immediately from \eqref{eq:displachn} and \eqref{eq:consist.equ}: for all $\vz_h$ in $\vU_h$,
\begin{align*}
&\frac{\alpha}{\Delta\,t} \int_\Om \big(\ve_u^n - \ve_u^{n-1}\big)\cdot \vz_h\\
&+ \frac{1}{E_1}\int_\Om \Big(\frac{\veps( \Pi_h\vu(t_n))}{1+ \lambda_1T_{\delta} \di\, \Pi_h\vu(t_{n-1})} -  \frac{\veps( \vu_h^n)}{1+ \lambda_1T_{\delta} \di\, \vu_h^{n-1}}\Big) : \veps( \vz_h) \\
& + \frac{|E_2|}{E_1} \int_\Om \Big[\frac{ 1+ \lambda_2T_{\delta} \di\, \Pi_h\vu(t_{n-1})}{( 1+ \lambda_1T_{\delta} \di\, \Pi_h\vu(t_{n-1}))}\frac{\di\,\Pi_h(\vu(t_n))}{F(T_{\delta} \di\, \Pi_h\vu(t_{n-1}))}\\
&- 
\frac{ 1+ \lambda_2T_{\delta} \di\, \vu_h^{n-1}}{( 1+ \lambda_1T_{\delta} \di\, \vu_h^{n-1})} \frac{\di\,\vu_h^n}{F(T_{\delta} \di\, \vu_h^{n-1})}\Big](\di\,\vz_h) \\
& 
 = \alpha \int_\Om  \ve_v^n \cdot \vz_h + R_{1,u}^n(\vz_h) + R_{2,u}^n(\vz_h),
\end{align*}
 where $R_{1,u}^n(\vz_h)$ and $R_{2,u}^n(\vz_h)$ are given respectively by \eqref{eq:R1u} and \eqref{eq:R2u}.

Next, we rewrite the non-linearities in terms of $\ve_u^n$. 
The first non-linear term is written
$$ \frac{1}{E_1}\int_\Om \frac{\veps(\ve_u^n) : \veps( \vz_h)}{1+ \lambda_1T_{\delta} \di\, \vu_h^{n-1}} + R_{3,u}^n(\vz_h), 
 $$
 where 
\begin{equation}
\label{eq:R3u}
R_{3,u}^n(\vz_h) := \frac{\lambda_1}{E_1}\int_\Om \frac{\veps( \Pi_h\vu(t_n)):\veps( \vz_h)\big(T_{\delta} \di\, \vu_h^{n-1} - T_{\delta} \di\, \Pi_h\vu(t_{n-1})\big) }{(1+ \lambda_1T_{\delta} \di\, \Pi_h\vu(t_{n-1})) (1+ \lambda_1T_{\delta} \di\, \vu_h^{n-1})}.
\end{equation}
The second non-linear term is written 
$$ \frac{|E_2|}{E_1} \int_\Om \frac{( \di\,\ve_u^n) (\di\,\vz_h) (1+ \lambda_2T_{\delta} \di\, \vu_h^{n-1})}{(1+ \lambda_1T_{\delta} \di\, \vu_h^{n-1})F(T_{\delta} \di\,\vu_h^{n-1})} +  R_{4,u}^n(\vz_h), 
 $$
where, after collecting terms, and denoting by $P_{\rm od} $ the product
$$
P_{\rm od} := (1+ \lambda_1T_{\delta} \di\, \vu_h^{n-1}) (1+ \lambda_1T_{\delta} \di \Pi_h\vu(t_{n-1}))F(T_{\delta} \di\,\vu_h^{n-1}) F(T_{\delta} \di\Pi_h\vu(t_{n-1})),   
$$
\begin{equation}
\label{eq:R4u}
\begin{split}
&R_{4,u}^n(\vz_h) := \frac{|E_2|}{E_1}\int_\Om \big(\di\,\Pi_h\vu(t_{n})\big)(\di\,\vz_h) \big(T_{\delta} \di\, \Pi_h\vu(t_{n-1})- T_{\delta} \di\, \vu_h^{n-1}\big)\\
& \times\Big[\frac{\lambda_2} {(1+ \lambda_1T_{\delta}\di\, \Pi_h\vu(t_{n-1}))F(T_{\delta}\di\, \Pi_h\vu(t_{n-1}))}\\
& - \frac{1}{P_{\rm od}}(1+ \lambda_2T_{\delta} \di\, \vu_h^{n-1}) \big(\gamma_4 
(1+ \lambda_1T_{\delta} \di\, \vu_h^{n-1}) + 
\lambda_1 F(T_{\delta} \di \Pi_h\vu_h^{n-1})\big)\Big].
\end{split}
\end{equation}
This leads to the following error equality for the displacement,  for all $\vz_h$ in $\vU_h$:
 \begin{equation}
\label{eq:error.equ}
\begin{split}
&\frac{\alpha}{\Delta\,t} \int_\Om \big(\ve_u^n - \ve_u^{n-1}\big)\cdot \vz_h + \frac{1}{E_1}\int_\Om \frac{\veps(\ve_u^n) : \veps( \vz_h)}{1+ \lambda_1T_{\delta} \di\, \vu_h^{n-1}}\\
& + \frac{|E_2|}{E_1}\int_\Om \frac{( \di\,\ve_u^n) (\di\,\vz_h) (1+ \lambda_2T_{\delta} \di\, \vu_h^{n-1})}{(1+ \lambda_1T_{\delta} \di\, \vu_h^{n-1})F(T_{\delta} \di\,\vu_h^{n-1})}\\
& 
 = \alpha \int_\Om  \ve_v^n \cdot \vz_h + R_{1,u}^n(\vz_h) + R_{2,u}^n(\vz_h)- R_{3,u}^n(\vz_h) - R_{4,u}^n(\vz_h).
\end{split}
 \end{equation}
 
An error equality for the fluid's velocity is simpler because it arises from a linear equation. It follows from \eqref{eq:fluidhn} and \eqref{eq:consist.eqv}: for all $\vw_h$ in $\vV_h$,
\begin{equation}
     \label{eq:erroreqv}
     \begin{split}
\frac{\varrho}{\Delta\,t}\int_\Om (\ve_v^n- \ve_v^{n-1})\cdot &\vw_h
+ \nu \int_\Om \nabla\, \ve_v^n : \nabla\, 
\vw_h    + \alpha \int_\Om \ve_v^n \cdot \vw_h \\
&= \frac{\alpha}{\Delta\,t}\int_\Om (\ve_u^n- \ve_u^{n-1})\cdot\vw_h + R_{1,v}^n(\vw_h,q_h) + R_{2,v}^n(\vw_h),
\end{split}
\end{equation}
where $R_{1,v}^n(\vw_h,q_h)$ and  $R_{2,v}^n(\vw_h)$ are given respectively by \eqref{eq:R1v} and \eqref{eq:R2v}.

\subsection{Upper bounds of the consistency terms}
\label{subsec:bddconsist}

For $R_{1,u}^n(\vz_h)$, with the help of relation \eqref{eq:difTdelt} and estimates \eqref{e:gam12}, \eqref{eq:gam56.1}, and \eqref{eq:gam78.1}, we obtain
\begin{align*}
|R_{1,u}^n(\vz_h)| &\lesssim (\Delta\,t)^{-1} \big|\langle  \int_{t_{n-1}}^{t_n} (\pdd_t\vu(s) - \pdd_t\vu(t_n)) \nd s, \vz_h\rangle\big| \\
& + c_\lambda\int_\Om |\veps( \vu(t_n))|\, |\veps( \vz_h)| \big|T_{2\delta} \di ( \vu(t_{n}) - \vu(t_{n-1}) )\big| \\
& + c_\lambda\int_\Om |\di\, \vu(t_{n})|\,|\di\, \vz_h|
\big|T_{2\delta} \di ( \vu(t_{n}) - \vu(t_{n-1}) )\big|. 
\end{align*}
But
\begin{equation*}
    \begin{split}
\frac{1}{\Delta\,t}  \int_{t_{n-1}}^{t_n} \| \pdd_t\vu(s) - \pdd_t\vu(t_n)\|_{H^{-1}(\Omega)} \nd s  &= \frac{1}{\Delta\,t}  \int_{t_{n-1}}^{t_n} \int_s^{t_n} \| \pdd_{tt}^2\vu(t)\|_{H^{-1}(\Omega)} \nd t\, \nd s   
\\
&\lesssim (\Delta\,t)^\frac{1}{2} \| \pdd_{tt}^2\vu\|_{L^2(t_{n-1},t_n;H^{-1}(\Omega))},
\end{split}
\end{equation*}
so that in view  of $\|\di\, \vu\|_{L^\infty(Q_T)} + \|\veps( \vu)\|_{L^\infty(Q_T)} < \infty$, see regularity assumption \eqref{additional_reg}, we get
\begin{equation*}
     \label{eq:bddR1u}
     \begin{split}
|R_{1,u}^n(\vz_h)| &\lesssim (\Delta\,t)^\frac{1}{2}\| \pdd_{tt}^2\vu\|_{L^2(t_{n-1},t_n; H^{-1}(\Om))} \|\nabla\,\vz_h\|_{L^2(\Om)}\\
& +c_\lambda (\Delta\, t)^\frac{1}{2}\big( \|\veps( \vz_h)\|_{L^2(\Om)} +\|\di\, \vz_h\|_{L^2(\Om)}\big)  \|\di\,\pdd_t \vu\|_{L^2(t_{n-1},t_n; L^2(\Om))}. 
     \end{split}
     \end{equation*}
     
Regarding $R_{2,u}^n(\vz_h)$, using relation \eqref{eq:difTdelt} again,  we have 
\begin{align*}
   & |R_{2,u}^n(\vz_h)| \le  \frac{\alpha}{\Delta\,t}\int_\Om \int_{t_{n-1}}^{t_n}\big|\pdd_t(\Pi_h\vu(t) - \vu(t)) \big|\,\nd t |\vz_h| +  \alpha \int_\Om  |\Pi_h\vv(t_n)-\vv(t_n) |\, |\vz_h|\\
 &+  \frac{|\lambda_1|}{E_1}\int_\Om |\veps( \Pi_h\vu(t_n))|\,|\veps( \vz_h)|  \frac{\big|T_{2\delta}(\di( \vu(t_{n-1}) - \Pi_h\vu(t_{n-1}))\big| }
{ (1+ \lambda_1T_{\delta} \di\, \Pi_\vu(t_{n-1})))(1+ \lambda_1T_{\delta} \di\, \vu(t_{n-1}))}\\
& + \frac{1}{E_1} \int_\Om \frac{|\veps(\Pi_h\vu(t_n)) - \vu(t_n))| }
{|1+ \lambda_1T_{\delta} \di\, \vu(t_{n-1})| } |\veps( \vz_h)| \\
&+  \frac{|E_2|}{E_1}\int_\Om  \Big[|\di( \Pi_h \vu(t_{n})- \vu(t_{n}))|\frac{  |1+ \lambda_2T_{\delta} \di\, \Pi_h \vu(t_{n-1})|}{|1+ \lambda_1T_{\delta} \di\, \Pi_h \vu(t_{n-1})|} \frac{|\di\,\vz_h|}{|F(T_{\delta} \di\, \Pi_h\vu(t_{n-1}))|}\\
&+  |\di\, \vu(t_{n})| |\di\,\vz_h|\Big|\frac{1+ \lambda_2T_{\delta} \di\, \Pi_h \vu(t_{n-1})}{(1+ \lambda_1T_{\delta} \di\, \Pi_h \vu(t_{n-1}))F(T_{\delta} \di\, \Pi_h\vu(t_{n-1}))}\\
&-
\frac{1+ \lambda_2T_{\delta} \di\, \vu(t_{n-1})}{(1+ \lambda_1T_{\delta} \di\,  \vu(t_{n-1}))F(T_{\delta} \di\, \vu(t_{n-1}))}\Big|\Bigg].
\end{align*}
Thus, by collecting terms, estimates \eqref{e:gam12}, \eqref{eq:gam56.1}, and \eqref{eq:gam78.1}, the regularity assumption \eqref{additional_reg} to estimate $\|\di\, \vu\|_{L^\infty(Q_T)}$, $\|\veps( \vu)\|_{L^\infty(Q_T)}$ and $\|\nabla^2\,\vu\|_{L^{\infty}(0,T;L^2(\Omega))}$, and the  approximation properties of $\Pi_h$ imply for sufficiently small $h$
\begin{equation*}
     \label{eq:bddR2u}
     \begin{split}
|R_{2,u}^n(\vz_h)| &\lesssim  h\big(\|\nabla\,\vv\|_{L^{\infty}(0,T; L^2(\Om))}+(\Delta\,t)^{-\frac{1}{2}} \|\nabla\,\pdd_t\vu\|_{L^2(t_{n-1},t_n; L^2(\Om))} \big) \|\vz_h\|_{L^2(\Om)}\\
& + h \Big( \|\veps( \vz_h)\|_{L^2(\Om)} +  \|\di\, \vz_h\|_{L^2(\Om)}\Big).
\end{split}
\end{equation*}

For $R_{3,u}^n(\vz_h)$, we easily derive
\begin{equation*}
\label{eq:bddR3u}
\begin{split}
|R_{3,u}^n(\vz_h) | &\le \frac{|\lambda_1|}{E_1}\gamma_2^2 C \|\veps( \vu)\|_{L^\infty(Q_T)}\|\di\, \ve_u^{n-1}\|_{L^2(\Om)} \|\veps( \vz_h)\|_{L^2(\Om)}\\
& \le c_\lambda \|\di\, \ve_u^{n-1}\|_{L^2(\Om)} \|\veps( \vz_h)\|_{L^2(\Om)},
\end{split}
\end{equation*}
and for $R_{4,u}^n(\vz_h)$, again invoking the regularity assumption \eqref{additional_reg} to estimate $\|\di\, \vu\|_{L^\infty(Q_T)}$,
\begin{equation*}
\label{eq:bddR4u}
\begin{split}
|R_{4,u}^n(\vz_h) |& \le c_\lambda\|\di\, \ve_u^{n-1}\|_{L^2(\Om)} \|\di\, \vz_h\|_{L^2(\Om)}.
\end{split}
 \end{equation*} 

With similar arguments, we estimate $R_{1,v}^n(\vw_h,q_h)$
\begin{equation*}
\label{eq:bddR1v}
\begin{split}
|R_{1,v}^n(\vw_h,q_h)| & \lesssim  h \|\nabla p\|_{L^{\infty}(0,T;L^2(\Om))} \|\di\,\vw_h\|_{L^2(\Om)}\\
& +  (\Delta\,t)^{\frac{1}{2}}  \|\nabla\, \vw_h\|_{L^2(\Om)} \big(  \| \pdd_{tt}^2 \vv\|_{L^2(t_{n-1},t_n; H^{-1}(\Om))} + \|\pdd_{tt}^2 \vu\|_{L^2(t_{n-1},t_n; H^{-1}(\Om))}\big).
\end{split}
\end{equation*}
Finally, $R_{2,v}^n(\vw_h)$ satisfies the upper bound
\begin{equation*}
\label{eq:bddR2v}
\begin{split}
|R_{2,v}^n(\vw_h)| &\lesssim h  (\Delta\,t)^{-\frac{1}{2}} \|\vw_h\|_{L^2(\Om)} \big( \|\nabla\, \pdd_t\vv\|_{L^2(t_{n-1},t_n; L^2(\Om))} + \|\nabla\, \pdd_t\vu\|_{L^2(t_{n-1},t_n; L^2(\Om))}\big)\\
& +  h \big(\|\nabla\,\vv\|_{L^{\infty}(0,T; L^2(\Om))}\|\vw_h\|_{L^2(\Om)}+\|\nabla^2\,\vv\|_{L^{\infty}(0,T; L^2(\Om))}\|\nabla\,\vw_h\|_{L^2(\Om)}\big).
\end{split}
\end{equation*}

\subsection{Preliminary error estimate for the displacement}
\label{subsec:1stbdddispl}

The choice $\vz_h =\ve_u^n $ in \eqref{eq:error.equ} together with the upper bounds obtained in the previous section lead to
\begin{equation}
\label{eq:1stbddeu}
\begin{split}
\frac{1}{\Delta\,t} \big( \|\ve_u^n\|^2_{L^2(\Om)} &- \|\ve_u^{n-1}\|^2_{L^2(\Om)} + \|\ve_u^n - \ve_u^{n-1}\|^2_{L^2(\Om)}\big) +   \|\veps(\ve_u^n)\|^2_{L^2(\Om)}  + \|\di\,\ve_u^n\|^2_{L^2(\Om)}\\
&\le c (\mathcal A+ \|\ve_v^n\|_{L^2(\Om)}) \|\ve_u^n\|_{L^2(\Om)} + c (\Delta\,t)^{\frac{1}{2}}\|\pdd_{tt}^2\vu\|_{L^2(t_{n-1},t_n; H^{-1}(\Om))}\|\nabla\ve_u^n\|_{L^2(\Omega)} \\
& + c\left((\Delta\,t)^{\frac{1}{2}}\|\di\,\pdd_t \vu\|_{L^2(t_{n-1},t_n; L^2(\Om))}+h\right)\left(\|\veps(\ve_u^n)\|_{L^2(\Om)}+\|\di\,\ve_u^n\|_{L^2(\Om)}\right) \\
&+ c_{\lambda} \|\di\,\ve_u^{n-1}\|_{L^2(\Om)} \left(\|\veps(\ve_u^n)\|_{L^2(\Om)}+  \|\di\,\ve_u^n\|_{L^2(\Om)}\right),
\end{split}
\end{equation}
where
$$
\mathcal A:= h\big(\|\nabla\,\vv\|_{L^{\infty}(0,T; L^2(\Om))} + (\Delta\,t)^{-\frac{1}{2}} \|\nabla \pdd_t\vu\|_{L^2(t_{n-1},t_n; L^2(\Om))}\big).
$$

We start by estimating the contributions of the nonlinear term, i.e., the term multiplying $c_\lambda$ is controlled by
\begin{equation}
    \label{eq:bddC5}
\frac{1}{4} \left( \|\veps(\ve_u^n)\|^2_{L^2(\Om)} + \|\di\,\ve_u^n\|^2_{L^2(\Om)}\right)+ 2c_\lambda^2\|\di\,\ve_u^{n-1}\|^2_{L^2(\Om)}.
\end{equation}
For the term with $\ve_u^n$ in factor, we invoke the Poincar\'e \eqref{eq:Poincare} and Korn \eqref{eq:Korn} inequalities  to write
\begin{equation}
\label{eq:bddeu}
\begin{split}
c\|\ve_u^n\|_{L^2(\Om)} (\mathcal A+ \|\ve_v^n\|_{L^2(\Om)})& \lesssim   \|\veps(\ve_u^n)\|_{L^2(\Om)} (\mathcal A + \|\ve_v^n\|_{L^2(\Om)})\\
&\le \frac{1}{8}  \|\veps(\ve_u^n)\|^2_{L^2(\Om)}+ 4 (\mathcal A^2+ \|\ve_v^n\|_{L^2(\Om)}^2).
\end{split}
\end{equation}
Finally, the second and third terms are estimated by
$$\frac{1}{8}\|\veps(\ve_u^n)\|_{L^2(\Omega)}^2+2(c\mathcal{K})^2\Delta\,t\|\pdd_{tt}^2\vu\|_{L^2(t_{n-1},t_n; H^{-1}(\Om))}^2
$$
and
$$\frac{1}{4} \left( \|\veps(\ve_u^n)\|^2_{L^2(\Om)} + \|\di\,\ve_u^n\|^2_{L^2(\Om)}\right)+4c^2\left(\Delta\,t\|\di\,\pdd_t \vu\|_{L^2(t_{n-1},t_n; L^2(\Om))}^2+h^2\right).$$

Returning to \eqref{eq:1stbddeu} multiplied by $\Delta \, t$, summing over $n$, using the additional regularity assumption~\eqref{additional_reg} and using that $\ve_u^0=0$, we obtain
\begin{align*}
    \|\ve_u^n\|^2_{L^2(\Om)} & + \sum_{i=1}^n \|\ve_u^i- \ve_u^{i-1}\|^2_{L^2(\Om)} + (\Delta\,t)\sum_{i=1}^n \left( \| \veps (\ve_u^i) \|_{L^2(\Omega)}^2 +  \| \di \, \ve_u^i \|_{L^2(\Omega)}^2\right)\\
& \lesssim   \sum_{i=1}^n \Delta\,t\|\ve_v^i\|^2_{L^2(\Om)} + c_\lambda\sum_{i=1}^n \Delta \,t\|\di\,\ve_u^{i-1}\|^2_{L^2(\Om)}+  h^2 + (\Delta\,t)^2.
\end{align*}
Assuming that $|\lambda_1|$ and $|\lambda_2|$ are sufficiently small to absorb the term multiplied by $c_\lambda$ on the left of the previous estimate, we arrive at
\begin{equation}\label{e:prelim_u}
\begin{split}
    \|\ve_u^n\|^2_{L^2(\Om)} &+ \sum_{i=1}^n \|\ve_u^i- \ve_u^{i-1}\|^2_{L^2(\Om)} + (\Delta\,t)\sum_{i=1}^n \left( \| \veps (\ve_u^i) \|_{L^2(\Omega)}^2 +  \| \di \, \ve_u^i \|_{L^2(\Omega)}^2\right)\\
& \lesssim  \sum_{i=1}^n (\Delta\,t)\|\ve_v^i\|^2_{L^2(\Om)} +  h^2 + (\Delta\,t)^2.
\end{split}
\end{equation}

Obviously, this requires an error estimate for the velocity of the fluid. Taking into account the stability estimates for the fluid derived in Section \ref{subsec:stab.h} such an error estimate will, in turn, require an error estimate for the discrete time derivative of the displacement in $H^{-1}(\Om)^d$. This is the object of the next subsection.

\subsection{Preliminary error estimate for the discrete time derivative of the displacement}
\label{subsec:1srbdd.dtu}

By definition,
$$
\|\ve_u^n- \ve_u^{n-1}\|_{H^{-1}(\Om)} = \sup_{\vf \in H^1_0(\Om)^d} \frac{1}{\|\nabla\,\vf\|_{L^2(\Om)}}\int_\Om (\ve_u^n- \ve_u^{n-1})\cdot \vf
$$
and we write
$$
\int_\Om (\ve_u^n- \ve_u^{n-1})\cdot \vf = \int_\Om (\ve_u^n- \ve_u^{n-1})\cdot (\vf- \Pi_h(\vf)) + \int_\Om (\ve_u^n- \ve_u^{n-1})\cdot  \Pi_h\vf,
$$
where $\Pi_h$ is an interpolation operator from $H^1_0(\Om)^d$ into $\vU_h$ satisfying \eqref{eq:PihC}. Therefore
\begin{equation*}
 \|\ve_u^n- \ve_u^{n-1}\|_{H^{-1}(\Om)} \lesssim  h  \|\ve_u^n- \ve_u^{n-1}\|_{L^2(\Om)}+ \sup_{\vf \in H^1_0(\Om)^d} \frac{1}{\|\nabla\,\vf\|_{L^2(\Om)}}\int_\Om (\ve_u^n- \ve_u^{n-1})\cdot \Pi_h\vf.
\end{equation*}
For the second term, we deduce from the error equality \eqref{eq:error.equ}, the estimates on $R_{i,u}^n$, $i=1,2,3,4$, obtained in Section~\ref{subsec:bddconsist}, and the stability \eqref{eq:PihC} of $\Pi_h$
\begin{align*}
\sup_{\vf \in H^1_0(\Om)^d} &\frac{1}{\|\nabla\,\vf\|_{L^2(\Om)}}\int_\Om (\ve_u^n- \ve_u^{n-1})\cdot \Pi_h\vf \\
& \lesssim (\Delta\, t) \left(\|\veps(\ve_u^n)\|_{L^2(\Om)} +  \|\di\,\ve_u^n\|_{L^2(\Om)}+  \|\ve_v^n\|_{L^2(\Om)} + c_\lambda\|\di\,\ve_u^{n-1}\|_{L^2(\Om)}\right)\\
& + (\Delta\,t)^{\frac{3}{2}}\|\pdd_{tt}^2 \vu\|_{L^2(t_{n-1},t_n; H^{-1}(\Om))} + (\Delta\,t)^{\frac{3}{2}} \|\di\,\pdd_t \vu\|_{L^2(t_{n-1},t_n; L^2(\Om))} + h (\Delta\, t)\\
& + h (\Delta\,t)^{\frac{1}{2}}\|\nabla\,\pdd_t \vu\|_{L^2(t_{n-1},t_n; L^2(\Om))} + h (\Delta\, t)\|\nabla\,\vv\|_{L^{\infty}(0,T; L^2(\Om))}.\\
    \end{align*}
Gathering the above two estimates, using that $\di\,\ve_u^{0}=0$,  and in view of the additional regularity assumption \eqref{additional_reg}, we get
\begin{equation*}
\begin{split}
\sum_{i=1}^n (\Delta \, t)^{-1} &\|\ve_u^i- \ve_u^{i-1}\|^2_{H^{-1}(\Om)} \lesssim  \sum_{i=1}^n(\Delta\, t) \|\ve_v^i\|^2_{L^2(\Om)} \\
& + h^2 (\Delta \, t)^{-1} \sum_{i=1}^n\|\ve_u^i- \ve_u^{i-1}\|_{L^2(\Om)}^2+ \sum_{i=1}^n(\Delta\, t) \left(\|\veps(\ve_u^i)\|^2_{L^2(\Om)}+ \|\di\,\ve_u^i\|^2_{L^2(\Om)}\right) \\
&  +  (\Delta\,t)^{2} + h^2.
\end{split}
    \end{equation*}
    It remains to take advantage of the assumption $h^2 \lesssim (\Delta \, t)$ and the previously derived estimate \eqref{e:prelim_u} on the displacement to deduce an estimate for the disctete time derivative
\begin{equation}
\label{e:estim_time_der}
\sum_{i=1}^n (\Delta \, t)^{-1} \|\ve_u^i- \ve_u^{i-1}\|^2_{H^{-1}(\Om)} \lesssim  \sum_{i=1}^n(\Delta\, t) \|\ve_v^i\|^2_{L^2(\Om)}  +  (\Delta\,t)^{2} + h^2.
    \end{equation}
    
\subsection{A priori estimate for the fluid's velocity}
\label{subsec:err.bddvel.fl}

We start from relation \eqref{eq:erroreqv} with $\vw_h = \ve_v^n$ and take advantage of the estimates for the consistency terms $R_{1,v}^n$, $R_{2,v}^n$ derived in Section~\ref{subsec:bddconsist}. We also apply Young's inequality and use the additional regularity assumption \eqref{additional_reg} to estimate $\|\nabla\,p\|_{L^{\infty}(0,T;L^2(\Omega))}$, $\|\nabla\,\vu\|_{L^{\infty}(0,T;L^2(\Omega))}$ as well as $\|\nabla^2\,\vu\|_{L^{\infty}(0,T;L^2(\Omega))}$ and obtain
\begin{align*}
&(\Delta\,t)^{-1} \big( \|\ve_v^n\|^2_{L^2(\Om)} - \|\ve_v^{n-1}\|^2_{L^2(\Om)} + \|\ve_v^n - \ve_v^{n-1}\|^2_{L^2(\Om)}\big) +  \|\nabla\,\ve_v^n\|^2_{L^2(\Om)} +  \|\ve_v^n\|^2_{L^2(\Om)}\\
& \le \frac{1}{2} \|\ve_v^n\|^2_{L^2(\Om)} + \frac 1 2 \|\nabla\,\ve_v^n\|^2_{L^2(\Om)} + (\Delta \,t)^{-2}\|\ve_u^n - \ve_u^{n-1}\|^2_{H^{-1}(\Om)}
 \\
& + (\Delta \,t)\| \pdd_{tt}^2 \vv\|^2_{L^2(t_{n-1},t_n; H^{-1}(\Om))} +(\Delta \,t) 
\|\pdd_{tt}^2 \vu\|^2_{L^2(t_{n-1},t_n; H^{-1}(\Om))}\\
&+ h^2 (\Delta\, t)^{-1}\|\nabla\, \pdd_t\vv\|^2_{L^2(t_{n-1},t_n; L^2(\Om))} + h^2 (\Delta\, t)^{-1} \|\nabla\, \pdd_t\vu\|^2_{L^2(t_{n-1},t_n; L^2(\Om))} + h^2. 
\end{align*}
Multiplying by $\Delta \,t$, summing over $n$, and recalling the estimate \eqref{e:estim_time_der} for the discrete time derivatives gives
\begin{equation}\label{e:derive_temp}
\begin{split}
    \|\ve_v^n\|^2_{L^2(\Om)} &+ \sum_{i=1}^n \|\ve_v^i - \ve_v^{i-1}\|^2_{L^2(\Om)}+ \sum_{i=1}^n (\Delta \,t) \left( \|\ve_v^i\|^2_{L^2(\Om)} +\|\nabla\,\ve_v^i\|^2_{L^2(\Om)}\right) \\
& \lesssim  \sum_{i=1}^n (\Delta \,t) \|\ve_v^i\|^2_{L^2(\Om)}  + h^2 + (\Delta \,t)^2.
\end{split}
\end{equation}

In order to complete the estimate for the fluid velocity, we apply Gronwall's lemma.
In this aim, we write
$$
\sum_{i=1}^n \Delta \,t \|\ve_v^i\|^2_{L^2(\Om)} \lesssim \Delta \,t \|\ve_v^n - \ve_v^{n-1}\|^2_{L^2(\Om)} + \sum_{i=1}^{n-1} \Delta \,t \|\ve_v^i\|^2_{L^2(\Om)}.
$$
Since $\Delta \,t$ is assumed to be sufficiently small, returning to \eqref{e:derive_temp} yields
\begin{equation}\label{e:prelim_v}
\begin{split}
    \|\ve_v^n\|^2_{L^2(\Om)} &+ \sum_{i=1}^{n-1} \|\ve_v^i - \ve_v^{i-1}\|^2_{L^2(\Om)}+ \sum_{i=1}^n (\Delta \,t) \left( \|\ve_v^i\|^2_{L^2(\Om)} +\|\nabla\,\ve_v^i\|^2_{L^2(\Om)}\right) \\
& \lesssim  \sum_{i=1}^{n-1} (\Delta \,t) \|\ve_v^i\|^2_{L^2(\Om)}  + h^2 + (\Delta \,t)^2
\end{split}
\end{equation}
and the error for the velocity \eqref{eq:3errvfl} now follows from an application of  Gronwall's lemma.
Moreover, substituting this upper bound into \eqref{e:prelim_u}, we immediately deduce the error estimate \eqref{eq:2err.disp} for the displacement.

\subsection{Consistency relation for the average fluid's pressure}
\label{subsec:err.press}

An error estimate for the discrete pressure $p_h$ is deduced from that of the average pressure $P_h$; see \eqref{e:Phn}. Let us first derive an error equality for the average pressure $P(t):=\int_0^t p(s)ds$. At any time $t_n$, we have
\begin{align*}
\int_\Om P(t_n) \di\,\vw_h &= \varrho \int_\Om \big(\vv(t_n) - \vv_{0,f}\big) \cdot \vw_h + \int_0^{t_n}\int_\Om \big(\nu  \nabla\,\vv(t) : \nabla\,\vw_h + \alpha \vv(t) \cdot \vw_h)\\
& -  \alpha \int_\Om \big( \vu(t_n) - \vu_{0,s}\big)\cdot \vw_h .
\end{align*}
This can be written
\begin{align*}
\int_\Om P(t_n) \di\,\vw_h &= \varrho \int_\Om \big(\vv(t_n) - \vv_{0,f}\big) \cdot \vw_h + \sum_{i=1} ^n \Delta\, t\int_\Om \big(\nu  \nabla\,\vv(t_i) : \nabla\,\vw_h + \alpha \vv(t_i) \cdot \vw_h)\\
& - \alpha \int_\Om \big( \vu(t_n) - \vu_{0,s}\big)\cdot \vw_h +R_{1,P}^n(\vw_h) ,
\end{align*}
with the first consistency term
\begin{equation}
    \label{eq:r1P}
    \begin{split}
 R_{1,P}^n(\vw_h) &:= \int_0^{t_n}\int_\Om \big(\nu  \nabla\,\vv(t) : \nabla\,\vw_h + \alpha \vv(t) \cdot \vw_h)\\
 & - \sum_{i=1} ^n\Delta \,t\int_\Om \big(\nu  \nabla\,\vv(t_i) : \nabla\,\vw_h + \alpha \vv(t_i) \cdot \vw_h\big).
 \end{split}
\end{equation}
The consistency equality is obtained by inserting the interpolants $\Pi_h$ (with the same notation) on $\vX_h$, $Q_h$ and on $\vU_h$, as follows:
\begin{equation}
    \label{eq:copnsist.P}
    \begin{split}
\int_\Om \Pi_hP(t_n) &\di\,\vw_h = \varrho \int_\Om \Pi_h(\vv(t_n) - \vv_{0,f}) \cdot \vw_h\\
&+ \sum_{i=1} ^n \Delta\,t\int_\Om \big(\nu  \nabla\,\Pi_h\vv(t_i) : \nabla\,\vw_h + \alpha \Pi_h\vv(t_i) \cdot \vw_h)\\
& - \alpha \int_\Om \Pi_h( \vu(t_n) - \vu_{0,s})\cdot \vw_h +R_{1,P}^n(\vw_h) + R_{2,P}^n(\vw_h), 
\end{split}
\end{equation}
with the second consistency term
\begin{equation}
    \label{eq:R2.P}
    \begin{split}
&R_{2,P}^n(\vw_h) := \int_\Om \big(\Pi_hP(t_n) - P(t_n)\big) \di\,\vw_h +  \varrho \int_\Om \big( \vv(t_n) - \vv_{0,f} - \Pi_h(\vv(t_n) - \vv_{0,f})\big) \cdot \vw_h \\
& + \sum_{i=1} ^n \Delta \, t\int_\Om \big(\nu  \nabla\big(\vv(t_i) - \Pi_h\vv(t_i)\big) : \nabla\,\vw_h + \alpha \big(\vv(t_i)-\Pi_h\vv(t_i)\big) \cdot \vw_h\big)\\
& - \alpha \int_\Om\big( \vu(t_n) - \vu_{0,s} - \Pi_h( \vu(t_n) - \vu_{0,s})\big)\cdot \vw_h. 
\end{split}
\end{equation}

Let us introduce the error for the average pressure $e_P^n := \Pi_hP(t_n) - P_h^n$.
Subtracting \eqref{eq:Phn} from \eqref{eq:copnsist.P}, we obtain the following instrumental relation for the average pressure
\begin{equation}
     \label{eq:errorPhn}
     \begin{split}
  \int_\Om e_P^n\di\,\vw_h = 
   &\varrho\int_\Om \ve_v^n\cdot \vw_h
 +    \sum_{i=1}^n\Delta\,t\int_\Om\big( \nu  \nabla\, \ve_v^i : \nabla\,
\vw_h  + \alpha  \ve_v^i \cdot \vw_h \big)
- \alpha\int_\Om \ve_u^n\cdot\vw_h\\
& +R_{1,P}^n(\vw_h) + R_{2,P}^n(\vw_h).
\end{split}
\end{equation}

\subsection{Upper bounds for the average pressure consistency terms}
\label{subsec:bddcons.press}
The first consistency term can be written in terms of the average operator $m_i(\cdot)$, defined in \eqref{eq:averag} and applied component-wise, as
\begin{equation}
    \label{eq:r1Pbis}
 R_{1,P}^n(\vw_h) = \sum_{i=1} ^n\Delta \,t\int_{\Om}\big(\nu\nabla(m_i(\vv) - \vv(t_i)):\nabla\,\vw_h + \alpha( m_i(\vv) - \vv(t_i))\cdot \vw_h\big).
\end{equation}
According to \eqref{eq:mif-f}, we have
$$
m_i(\vv) - \vv(t_i) = - \frac{1}{\Delta\,t} \int_{t_{i-1}}^{t_i} \int_{t}^{t_i} \pdd_t \vv(s)\, \nd s\, \nd t
= - \frac{1}{\Delta\,t} \int_{t_{i-1}}^{t_i} (s- t_{i-1}) \pdd_t \vv(s) \nd s
$$
and thus
$$
|m_i(\vv) - \vv(t_i)| \lesssim (\Delta\,t)^{\frac{1}{2}}\|\pdd_t \vv\|_{L^2(t_{n-1},t_n)}.
$$
This together with the regularity assumption ~\eqref{additional_reg} and Poincar\'e's inequality \eqref{eq:Poincare} lead to the following estimate for the first consistency term
\begin{equation}
    \label{eq:bddr1P}
| R_{1,P}^n(\vw_h)| \lesssim (\Delta \,t) t_n^{\frac{1}{2}}\|\pdd_t \vv\|_{L^2(0,t_n;H^1(\Om))}\|\nabla\,\vw_h\|_{L^2(\Om)}\lesssim (\Delta \,t) t_n^{\frac{1}{2}}\|\nabla\,\vw_h\|_{L^2(\Om)}.
\end{equation}

Regarding the second consistency term, the approximation properties of the interpolation operators $\Pi_h$ readily yield
\begin{align*}
&\frac{1}{\|\nabla\,\vw_h\|_{L^2(\Om)}}|R_{2,P}(\vw_h)| \lesssim h \|\nabla\,P(t_n)\|_{L^2(\Om)} + h\|\nabla\,\vv(t_n)\|_{L^2(\Om)}  + h\|\nabla\,\vv_{0,f}\|_{L^2(\Om)} \\
& + h \sum_{i=1}^n \Delta\,t \big(\|D^2\,\vv(t_i)\|_{L^2(\Om)} + \|\nabla\,\vv(t_i)\|_{L^2(\Om)} \big)  
 + h\|\nabla\,\vu(t_n)\|_{L^2(\Om)}  + h\|\nabla\,\vu_{0,s}\|_{L^2(\Om)}.
\end{align*}
Hence, thanks to the regularity assumption~\eqref{additional_reg}, we obtain
\begin{equation}
    \label{eq:bddr2P}
| R_{2,P}(\vw_h)| \lesssim (1+t_n)h \|\nabla\,\vw_h\|_{L^2(\Om)}.
\end{equation}

\subsection{Error bound for the pressure}
\label{subsec:err.aver.press}

We first note that  the inf-sup condition \eqref{e:discrete_infsup} guarantees the existence of $\vw_h^n \in \vX_h$ such that
$$
\| e_P^n \|_{L^2(\Omega)}^2 \lesssim \int_\Omega \di\; (\vw_h^n) e_P^n, \qquad \| \nabla \vw_h^n\|_{L^2(\Om)} \lesssim \| e_P^n\|_{L^2(\Omega)}.
$$
This $\vw_h^n$ in the error relation \eqref{eq:errorPhn} gives
\begin{equation}
    \label{eq:errePn}
\begin{split}
 \|e_P^n\|_{{L^2(\Om)}}  \lesssim &\|\ve_v^n\|_{L^2(\Om)} + \Big( \big(\sum_{i=1}^n(\Delta\,t)\|\nabla\, \ve_v^i\|^2_{L^2(\Om)}\big)^{\frac{1}{2}} +\big(\sum_{i=1}^n (\Delta\,t)\|\ve_v^i\|^2_{L^2(\Om)}\big)^{\frac{1}{2}}\Big)\\
& + \|\ve_u^n\|_{L^2(\Om)} + (\Delta \,t) + h. 
\end{split}    
\end{equation}
By substituting~\eqref{e:prelim_v} and~\eqref{e:prelim_u} into \eqref{eq:errePn} we readily infer that, for all $n$, $1 \le n \le N$,
\begin{equation}
    \label{eq:2errePn}
\|e_P^n\|_{{L^2(\Om)}}  \lesssim\big(h^2 + (\Delta \,t)^2\big)^{\frac{1}{2}}\Big( 1+\text{exp}(c t_n)\Big).
\end{equation}
Notice that
$$
\Pi_hP(t_n) = \int_0^{t_n} \Pi_h p(t),
$$
and thus 
$$
e_P^n = \sum_{i=1}^n(e_P^i - e_P^{i-1}) = \sum_{i=1}^n (\Delta\, t)  e_p^i.
$$
Then, proceeding as in the proof of Proposition \ref{pro:bddpres}, we obtain for any function $\varphi$ in $W^{1, 1}_0(0,T;L^2_0(\Om))$,
$$
\big| \int_0^{T} \int_\Om e_p(t) \varphi(t) \big| 
 \le 2 \sum_{i=1}^{N-1} \| e_P^i\|_{{L^2(\Om)}} \|\pdd_t\varphi \|_{L^1(t_{i},t_{i+1}; {L^2(\Om)})} + \sum_{i=1}^{N} \| e_P^i\|_{{L^2(\Om)}} \|\pdd_t\varphi \|_{L^1(t_{i-1},t_i; {L^2(\Om)})}
$$
and the desired error estimate for the pressure \eqref{eq:} follows.

%--------------------------------
\section{Numerical illustration}
\label{sec:num_res}

In this section, we present several numerical experiments for the fully discretized problem \eqref{eq:initial}-\eqref{eq:displachn}-\eqref{eq:fluidhn}. All the tests are done in a two-dimensional domain discretized using a uniform mesh made of square elements. For the finite element spaces, $(\mathbb{Q}_2)^2$ is used for $\vU_h$ (solid's displacement) and the Taylor--Hood element $(\mathbb{Q}_2)^2-\mathbb{Q}_1$ is used for $\vX_h$ and $Q_h$ (fluid's velocity and pressure). The implementation is based on  \emph{deal.ii} \cite{dealii2019design} and Paraview \cite{ayachit2015} is used for the visualization.

\subsection{Manufactured solutions}
\label{subsec:manufactured}

We start by considering tests with a manufactured solution to numerically analyze the decay rates for the different errors. 
The spatial domain is the unit square $\Omega=(0,1)^2$, the final time is $T=1$, and the parameters are set to
$$\alpha=1, \,\, \nu=1, \,\, \varrho=1, \,\, \lambda_1=\lambda_2=2, \,\, E_1=1, \,\, E_2=-0.2.$$
The truncation parameter $\delta$ is chosen to be $0.4$ in order for the truncation not to be active.
We design the numerical experiment in such a way that the solution to the coupled system is given by
\begin{equation} \label{eq:manufac_u}
\vu(t,\vx) = \frac{\theta(t)}{100}\left(\begin{array}{c}
        \cos(2\pi x_1)\sin(2\pi x_2) \\ \sin(2\pi x_1)\cos(2\pi x_2)
\end{array}\right)
\end{equation}
and
\begin{equation} \label{eq:manufac_vp}
\vv(t,\vx) = \theta(t)\left(\begin{array}{c}
        \sin(2\pi x_1)\cos(2\pi x_2) \\ -\cos(2\pi x_1)\sin(2\pi x_2)
\end{array}\right), \quad p(\vx,t)=\theta(t)(60x_1^2x_2-20x_2^3-5),
\end{equation}
for $\vx=(x_1,x_2)\in\overline{\Omega}$, $t\in[0,T]$, and where
\begin{equation}\label{eq:factor_t}
    \theta(t)=\theta_1(t):=\exp(-t) \qquad \text {or} \qquad \theta(t)=\theta_2(t):=\sin(3\pi t).
\end{equation}
We refer to \cite[Section 6.2.1]{GGHR2} for a numerical study in the linear setting $\lambda_1=\lambda_2=0$ and the exact solutions corresponding to $\theta=\theta_1$.
Note that for the above functions to satisfy the coupled problem, we include a forcing term to the solid equation \eqref{eq:eps=T} and the fluid momentum equation \eqref{eq:bal}. Furthermore, non-homogeneous Dirichlet boundary conditions are prescribed for the solid displacement and the fluid velocity on $\partial\Omega$.  

For the time discretization, we use $N$ uniform snapshots in time corresponding to a time-step $\Delta\,t= T/N$. Given a level of refinement $m$, the subdivision of $\Omega$ consists of $2^{2m}$ square elements of diameter $h=\sqrt{2}2^{-m}$. Finally, we focus on the following errors
$$\text{err}_{\star}^{H_0^1}:= \sqrt{\sum_{n=1}^N\Delta\,t\|\nabla\text{err}_{\star}^n\|_{L^2(\Omega)}^2}, \quad \text{err}_{\star}^{L^2}:= \max_{1\le n\le N}\|\text{err}_{\star}^n\|_{L^2(\Omega)} \quad \text{for} \quad \star\in\{u,v\}$$
and 
$$\text{err}_p^{L^2}:=\sqrt{\sum_{n=1}^N\Delta\,t\|\text{err}_p^n\|_{L^2(\Omega)}^2},$$
where $\text{err}_u^n:=\vu(t_n)-\vu_h^n$, $\text{err}_v^n:=\vv(t_n)-\vv_h^n$ and $\text{err}_p^n:=p(t_n)-p_h^n$, $n=1,2,\ldots,N$.

Recall that Theorem~\ref{thm:velerr} is derived for finite element spaces satisfying \eqref{e:ass_error} (for ease of presentation). 
However, the finite element discretizations used in the numerical illustrations are expected to  yield
$$
\text{err}_{u}^{H_0^1}+\text{err}_{v}^{H_0^1} + \text{err}_p^{L^2} \in \mathcal O(h^2+\Delta t).
$$
We observe that the spatial error is dominant in the case $\theta=\theta_1$, while the temporal error dominates when $\theta=\theta_2$. In both cases, the error behaves as described above. The results for selected values of $N$ and $m$ are presented in Tables~\ref{tab:manufac_theta1} and~\ref{tab:manufac_theta2} for $\theta=\theta_1$ and $\theta=\theta_2$, respectively.
In Table~\ref{tab:errors_rate}, we report the convergence rates for $\text{err}_u^{H_0^1}$, $\text{err}_v^{H_0^1}$ and $\text{err}_p^{L^2}$ when $\Delta\,t= 12.8~h^2$ and observe that they all belong to  $\mathcal O(h^2)$ as expected.

\begin{table}[htbp]
    \centering
    \begin{tabular}{|c|c|c|c|c|c|c|}
    \hline
    $m$ & $N$ & $\text{err}_u^{H_0^1}$ & $\text{err}_u^{L^2}$ & $\text{err}_v^{H_0^1}$ & $\text{err}_v^{L^2}$ & $\text{err}_p^{L^2}$ \\
\hline
    
\hline
4 & 10 & 3.8349$\cdot 10^{-4}$ & 4.3634$\cdot 10^{-5}$ & 2.2632$\cdot 10^{-2}$ & 3.6360$\cdot 10^{-4}$ & 9.0339$\cdot 10^{-3}$ \\
4 & 40 & 2.4809$\cdot 10^{-4}$ & 1.1592$\cdot 10^{-5}$ & 2.3446$\cdot 10^{-2}$ & 3.3820$\cdot 10^{-4}$ & 9.2753$\cdot 10^{-3}$ \\
4 & 80 & 2.3975$\cdot 10^{-4}$ & 6.6330$\cdot 10^{-6}$ & 2.3590$\cdot 10^{-2}$ & 3.4215$\cdot 10^{-4}$ & 9.3277$\cdot 10^{-3}$ \\
\hline
5 & 10 & 3.1494$\cdot 10^{-4}$ & 4.3485$\cdot 10^{-5}$ & 5.9192$\cdot 10^{-3}$ & 2.1187$\cdot 10^{-4}$ & 2.6269$\cdot 10^{-3}$ \\
5 & 40 & 9.9120$\cdot 10^{-5}$ & 1.1055$\cdot 10^{-5}$ & 5.8790$\cdot 10^{-3}$ & 6.7737$\cdot 10^{-5}$ & 2.3428$\cdot 10^{-3}$ \\
5 & 80 & 7.1384$\cdot 10^{-5}$ & 5.6575$\cdot 10^{-6}$ & 5.9028$\cdot 10^{-3}$ & 4.9307$\cdot 10^{-5}$ & 2.3372$\cdot 10^{-3}$ \\
\hline
6 & 10 & 3.1017$\cdot 10^{-4}$ & 4.3481$\cdot 10^{-5}$ & 2.2807$\cdot 10^{-3}$ & 2.0976$\cdot 10^{-4}$ & 1.4950$\cdot 10^{-3}$ \\
6 & 40 & 8.1244$\cdot 10^{-5}$ & 1.1044$\cdot 10^{-5}$ & 1.5342$\cdot 10^{-3}$ & 5.6185$\cdot 10^{-5}$ & 6.7734$\cdot 10^{-4}$ \\
6 & 80 & 4.2785$\cdot 10^{-5}$ & 5.6385$\cdot 10^{-6}$ & 1.4922$\cdot 10^{-3}$ & 2.8774$\cdot 10^{-5}$ & 6.0866$\cdot 10^{-4}$ \\
\hline
\end{tabular}
    \caption{Errors for the manufactured solution with $\theta=\theta_1$ in the case $\lambda_1=\lambda_2=2$.}
    \label{tab:manufac_theta1}
\end{table}

\begin{table}[htbp]
    \centering
    \begin{tabular}{|c|c|c|c|c|c|c|}
    \hline
    $m$ & $N$ & $\text{err}_u^{H_0^1}$ & $\text{err}_u^{L^2}$ & $\text{err}_v^{H_0^1}$ & $\text{err}_v^{L^2}$ & $\text{err}_p^{L^2}$ \\
\hline

\hline
4 & 20 & 1.9051$\cdot 10^{-3}$ & 2.4200$\cdot 10^{-4}$ & 8.9573$\cdot 10^{-2}$ & 1.0454$\cdot 10^{-2}$ & 6.7303$\cdot 10^{-2}$ \\
4 & 80 & 5.4559$\cdot 10^{-4}$ & 6.4431$\cdot 10^{-5}$ & 3.3581$\cdot 10^{-2}$ & 2.7066$\cdot 10^{-3}$ & 1.9610$\cdot 10^{-2}$ \\
4 & 160 & 3.5286$\cdot 10^{-4}$ & 3.2828$\cdot 10^{-5}$ & 2.7770$\cdot 10^{-2}$ & 1.4023$\cdot 10^{-3}$ & 1.3124$\cdot 10^{-2}$ \\
\hline
5 & 20 & 1.8875$\cdot 10^{-3}$ & 2.4169$\cdot 10^{-4}$ & 8.6103$\cdot 10^{-2}$ & 1.0423$\cdot 10^{-2}$ & 6.6642$\cdot 10^{-2}$ \\
5 & 80 & 4.8442$\cdot 10^{-4}$ & 6.4032$\cdot 10^{-5}$ & 2.2736$\cdot 10^{-2}$ & 2.6595$\cdot 10^{-3}$ & 1.7058$\cdot 10^{-2}$ \\
5 & 160 & 2.4960$\cdot 10^{-4}$ & 3.2359$\cdot 10^{-5}$ & 1.2661$\cdot 10^{-2}$ & 1.3342$\cdot 10^{-3}$ & 8.8157$\cdot 10^{-3}$ \\
\hline
6 & 20 & 1.8864$\cdot 10^{-3}$ & 2.4167$\cdot 10^{-4}$ & 8.5881$\cdot 10^{-2}$ & 1.0421$\cdot 10^{-2}$ & 6.6600$\cdot 10^{-2}$ \\
6 & 80 & 4.8035$\cdot 10^{-4}$ & 6.4010$\cdot 10^{-5}$ & 2.1880$\cdot 10^{-2}$ & 2.6575$\cdot 10^{-3}$ & 1.6886$\cdot 10^{-2}$ \\
6 & 160 & 2.4170$\cdot 10^{-4}$ & 3.2336$\cdot 10^{-5}$ & 1.1051$\cdot 10^{-2}$ & 1.3319$\cdot 10^{-3}$ & 8.4743$\cdot 10^{-3}$ \\
\hline
    \end{tabular}
    \caption{Errors for the manufactured solution with $\theta=\theta_2$ in the case $\lambda_1=\lambda_2=2$.}
    \label{tab:manufac_theta2}
\end{table}

\begin{table}[htbp]
    \centering
    \begin{tabular}{|c|c|c|c|c|c|c|c|}
    \cline{3-8}
    \multicolumn{2}{c}{ } & \multicolumn{6}{|c|}{$\theta=\theta_1$} \\
    \hline
    $m$ & $N$ & $\text{err}_u^{H_0^1}$ & rate & $\text{err}_v^{H_0^1}$ & rate & $\text{err}_p^{L^2}$ & rate \\
    \hline
    4 & 10 & 3.8349$\cdot 10^{-4}$ & -- & 2.2632$\cdot 10^{-2}$ & -- & 9.0339$\cdot 10^{-3}$ & -- \\ 
    5 & 40 & 9.9120$\cdot 10^{-5}$ & 1.952 & 5.8790$\cdot 10^{-3}$ & 1.945 & 2.3428$\cdot 10^{-3}$ & 1.947 \\
    6 & 160 & 2.4992$\cdot 10^{-5}$ & 1.988 & 1.4837$\cdot 10^{-3}$ & 1.986 & 5.9101$\cdot 10^{-4}$ & 1.987 \\
    \hline
    \end{tabular} 

\vspace*{0.4cm}

    \begin{tabular}{|c|c|c|c|c|c|c|c|}
    \cline{3-8}
    \multicolumn{2}{c}{ } & \multicolumn{6}{|c|}{$\theta=\theta_2$} \\
    \hline
    $m$ & $N$ & $\text{err}_u^{H_0^1}$ & rate & $\text{err}_v^{H_0^1}$ & rate & $\text{err}_p^{L^2}$ & rate \\
    \hline
    4 & 10 & 3.7705$\cdot 10^{-3}$ & -- & 1.6862$\cdot 10^{-1}$ & -- & 1.3004$\cdot 10^{-1}$ & -- \\    
    5 & 40 & 9.5486$\cdot 10^{-4}$ & 1.981 & 4.3897$\cdot 10^{-2}$ & 1.942 & 3.3710$\cdot 10^{-2}$ & 1.948 \\
    6 & 160 & 2.4170$\cdot 10^{-4}$ & 1.982 & 1.1051$\cdot 10^{-2}$ & 1.990 & 8.4743$\cdot 10^{-3}$ & 1.992 \\
    \hline
    \end{tabular} 
    \caption{Rate of convergence for the solid displacement and fluid velocity errors in $L^2(0,T;H_0^1(\Omega)^2)$ and the pressure error $L^2(0,T;L^2(\Omega))$ in the case $\lambda_1=\lambda_2=2$ for both cases $\theta=\theta_1$ and $\theta=\theta_2$. The time-step is fixed to $\Delta\,t=12.8~h^2$.}
    \label{tab:errors_rate}
\end{table}

\subsection{Practical problem}
\label{subsec:physic}

In this section, we consider a variation of the problem analyzed numerically  in \cite[Section 6.2.1]{GGHR2} for $\lambda_1=\lambda_2=0$; see also \cite[Section 6.1]{GGHR1} for a similar test case. The computational domain is $\Omega=(0,2)\times(0,1)$ with boundary $\Gamma_1\cup\Gamma_2\cup\Gamma_3\cup\Gamma_4$ labeled counter-clockwise starting with the bottom part $\{(x_1,0) \ : \ x_1\in [0,2] \}$. The parameters are set to
$$T=1, \,\, \alpha=1, \,\, \nu=0.1, \,\, \varrho = 1 \,, E_1=\frac{1}{8}, \,\, E_2=-\frac{3}{64}$$
and the truncation parameter is large enough so that the truncation does not occur for all choices of $\lambda_1$, $\lambda_2$. We note that the restrictions on $\lambda_1$ and $\lambda_2$ in Section~\ref{sec:trunc} are not necessarily satisfied. However, during computation, we verify that the constants $B_1$ and $B_2$ in \eqref{eqn:B1B2} are positive.
Observe that the values of $E_1$ and $E_2$ correspond to $\mu=4$ and $\lambda=12$ for the Lam\'e parameters in \eqref{eqn:Ei_Lame} when $\lambda_1=\lambda_2=0$ (linear case). The boundary conditions for the fluid velocity are
\begin{equation} \label{eqn:BC_fluid_Example2}
    \vv = \mathbf{0} \quad \text{on } \Gamma_1\cup\Gamma_3 \qquad \text{and} \qquad (2\nu\veps(\vv)-pI)\vn=\mathbf{0} \quad \text{on } \Gamma_2\cup\Gamma_4
\end{equation}
while for the solid, we impose
$$\vu = \mathbf{0} \quad \text{on } \Gamma_2\cup\Gamma_3\cup\Gamma_4 \qquad \text{and} \qquad \vu(x_1,x_2)=(0,-\sin(0.5\pi x_1))^T \quad \text{on } \Gamma_1.$$
Note that due to the boundary conditions \eqref{eqn:BC_fluid_Example2}, the pressure is in $L^2(\Omega)$ not necessarily with vanishing mean value.
The external force $\varrho_s\mathbf{b}_e$ with $\varrho_s=2$ and $\mathbf{b}_e=(0,-1)^T$ is added to the right-hand side of equation \eqref{eq:bal}, while $\varrho\mathbf{b}_e$ is added to the right-hand side of the first equation of \eqref{eq:flow}. The initial conditions are set to $\vu(0)=\vv(0) = \mathbf{0}$.

Regarding the discretization parameters, the domain is partitioned using a uniform mesh made of 8192 squares of diameter $\sqrt{2}2^{-6}$, we use again $(\mathbb{Q}_2)^2-(\mathbb{Q}_2)^2-\mathbb{Q}_1$ for the finite element spaces, and the time-step is set to $\Delta\, t=0.1$.

Figure~\ref{fig:Example2_eps} reports the displacement, the velocity and the pressure at the final time $T=1$ for $\lambda_1=\lambda_2=5$ and we provide in Table~\ref{tab:epsu_linf} the evolution of $\|\veps(\vu_h^n)\|_{L^{\infty}(\Omega)}$ for different values of the model parameters $\lambda_1$ and $\lambda_2$. We observe an increase of more $90\%$ for $\|\veps(\vu_h^N)\|_{L^{\infty}(\Omega)}$ when $\lambda_1=\lambda_2=5$ compared to the linear case $\lambda_1=\lambda_2=0$. Refer to Figure~\ref{fig:Example2_strain} for a comparison of the strain when $\lambda_1=\lambda_2\in\{0,1,2,3,4,5\}$.

\begin{figure}
    \centering
    \includegraphics[trim={7cm 7cm 7cm 7cm},clip,width=0.325\linewidth]{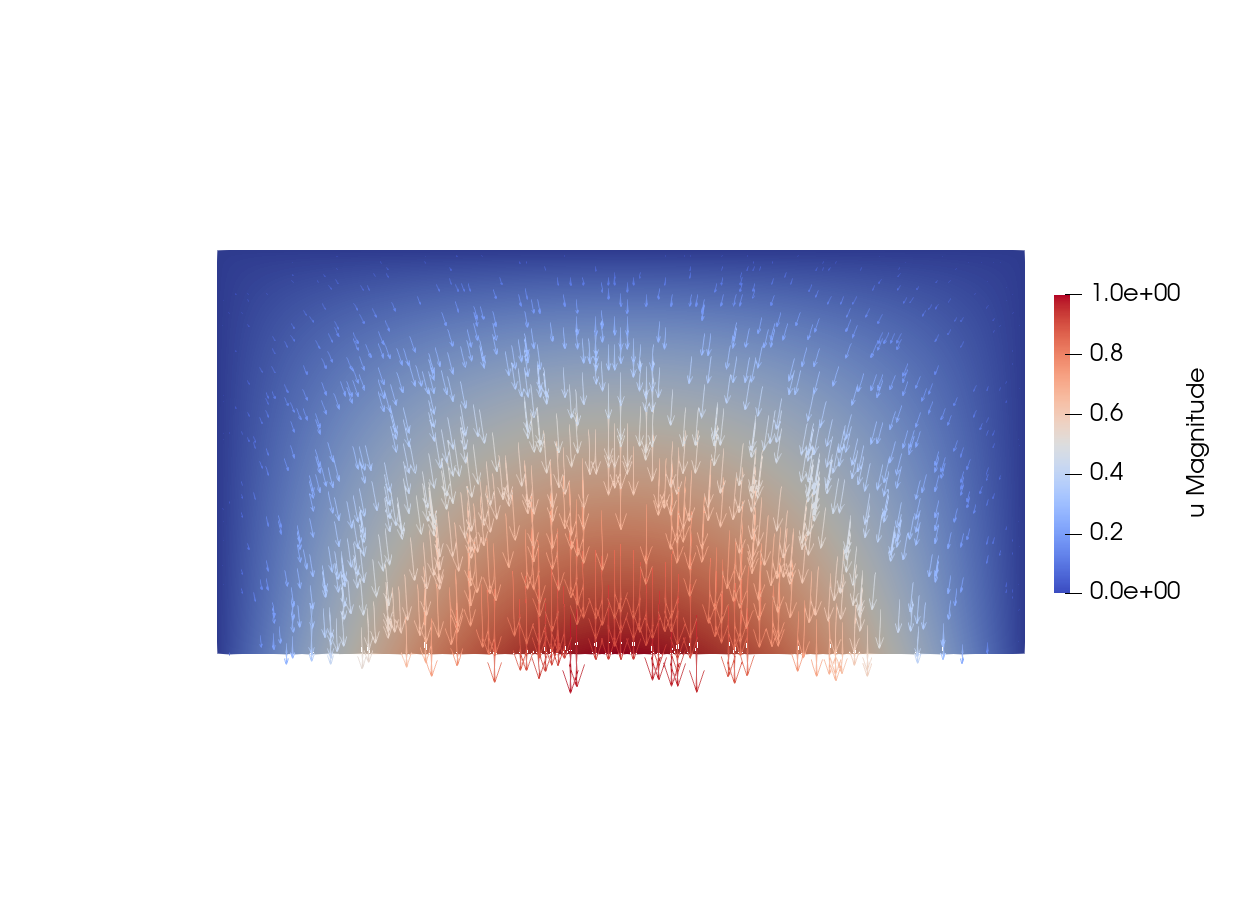} 
    \includegraphics[trim={7cm 7cm 7cm 7cm},clip,width=0.325\linewidth]{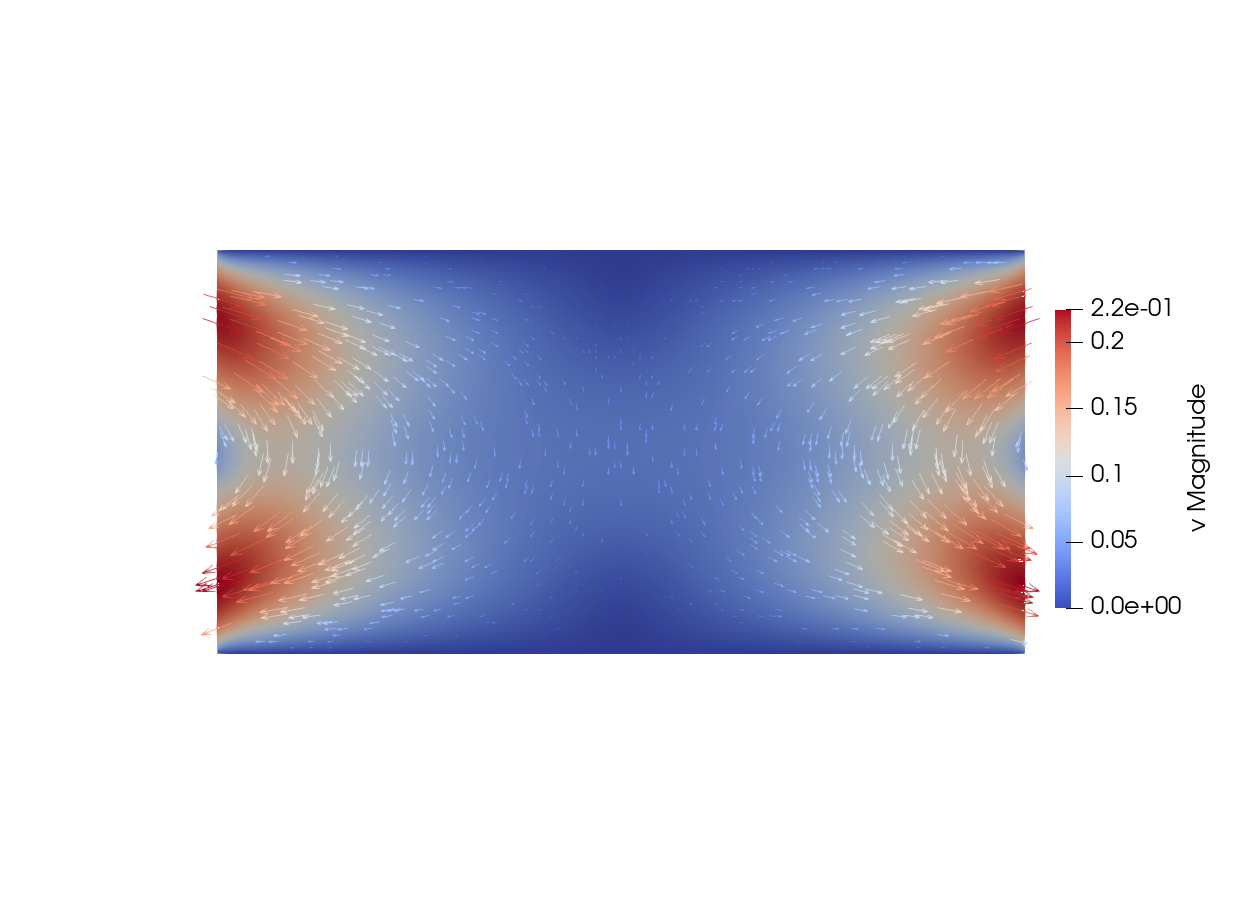} 
    \includegraphics[trim={7cm 7cm 7.2cm 7cm},clip,width=0.325\linewidth]{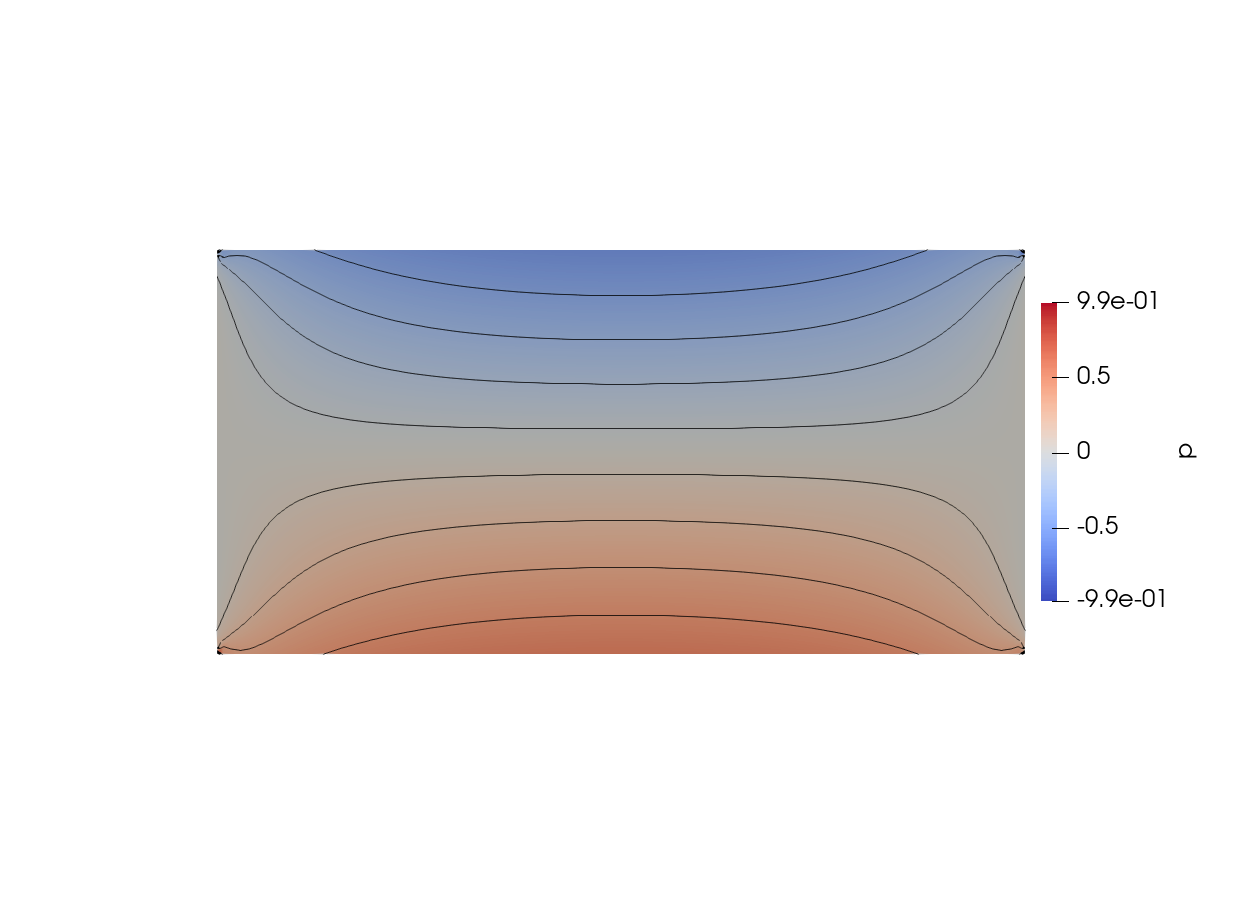}
    \caption{Solution at the final time $T=1$ using $\Delta\, t=0.1$ for the case $\lambda_1=\lambda_2=5$. Left: solid displacement $\vu$ (magnitude and vector field with dark blue=0.0 and dark red=1.0); middle: fluid velocity $\vv$ (magnitude and vector field  with dark blue=0.0 and dark red=0.22); right: fluid pressure $p$ with isolines using 16 values uniformly distributed between $\max_{\overline{\Omega}}p=-0.9868$ (dark blue) and $\max_{\overline{\Omega}}p=0.9945$ (dark red).}
    \label{fig:Example2_eps}
\end{figure}

\begin{table}[htbp]
    \centering
    \begin{tabular}{|c|r|r|r|r|r|r|}
    \cline{2-7}
    \multicolumn{1}{c|}{ } & \multicolumn{6}{|c|}{$\lambda_1=\lambda_2$} \\
    \hline
    $t_n$ & \multicolumn{1}{|c|}{0} & \multicolumn{1}{|c|}{1} & \multicolumn{1}{|c|}{2} & \multicolumn{1}{|c|}{3} & \multicolumn{1}{|c|}{4} & \multicolumn{1}{|c|}{5} \\
    \hline
0.1 & 1.15251 & 1.15251 & 1.15251 & 1.15251 & 1.15251 & 1.15251 \\
0.2 & 1.10503 & 1.09503 & 1.10053 & 1.09282 & 1.07514 & 1.07498 \\
0.4 & 1.10344 & 1.07718 & 1.07858 & 1.10469 & 1.21638 & 1.32958 \\
0.6 & 1.10361 & 1.07726 & 1.07899 & 1.22575 & 1.42431 & 1.63958 \\
0.8 & 1.10371 & 1.07728 & 1.07912 & 1.30193 & 1.58267 & 1.90757 \\
1 & 1.10375 & 1.07728 & 1.08778 & 1.35043 & 1.70443 & 2.14054 \\
\hline
    \end{tabular}
    \caption{Value of $\|\veps(\vu)\|_{L^{\infty}(\Omega)}$ at different times $0\le t_n\le T=1$ and different values of $\lambda_1=\lambda_2$.}
    \label{tab:epsu_linf}
\end{table}

\begin{figure}
    \centering
    \includegraphics[trim={7cm 7cm 7cm 7cm},clip,width=0.325\linewidth]{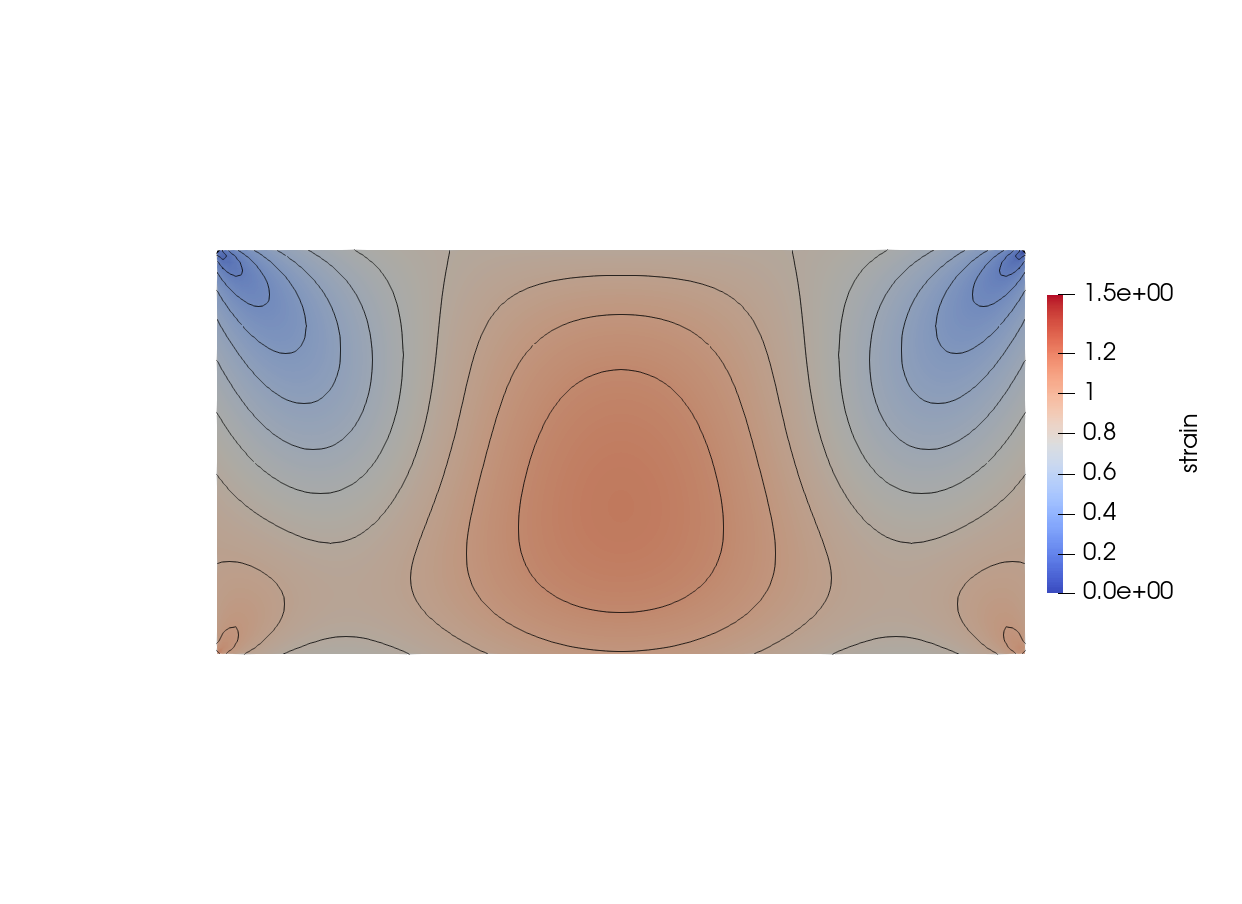}
    \includegraphics[trim={7cm 7cm 7cm 7cm},clip,width=0.325\linewidth]{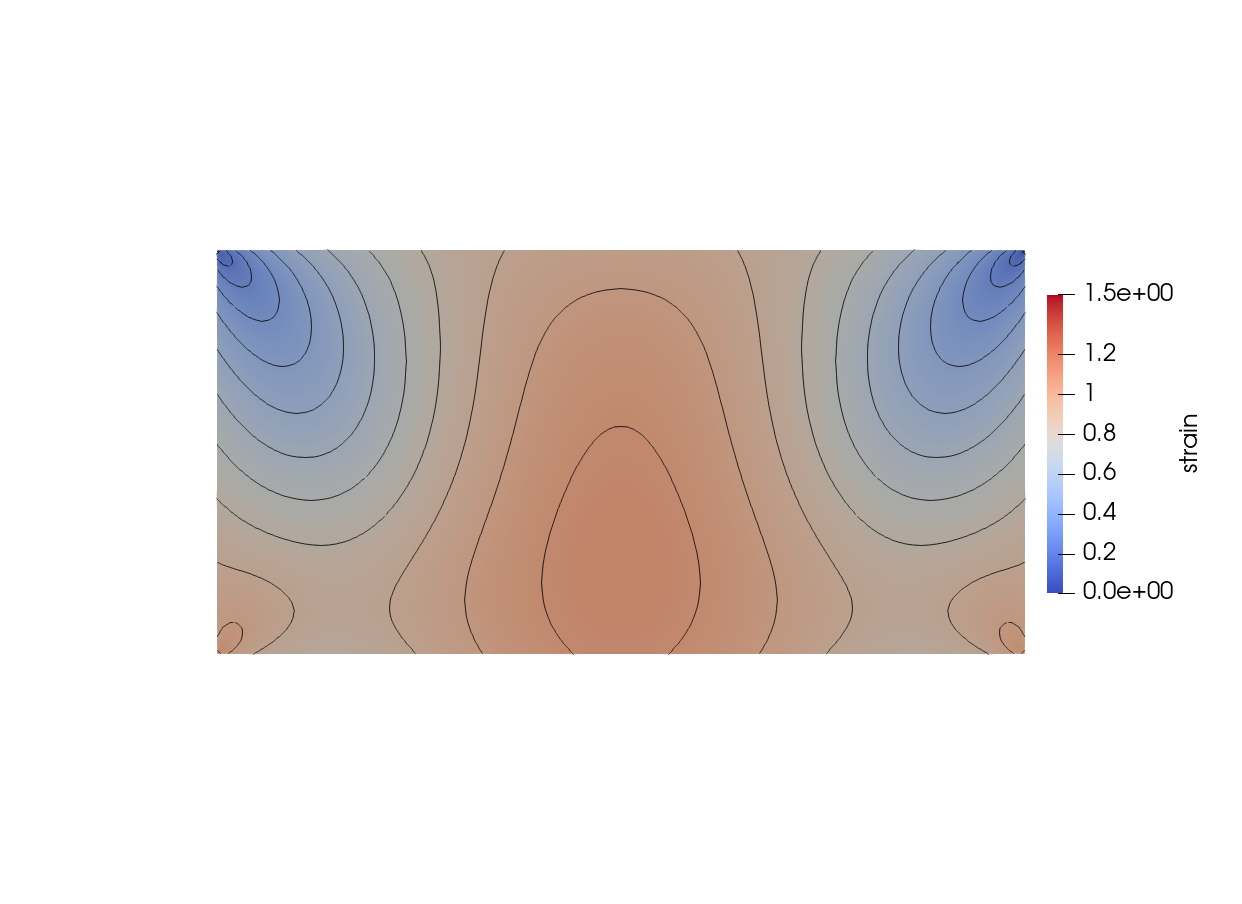}
    \includegraphics[trim={7cm 7cm 7cm 7cm},clip,width=0.325\linewidth]{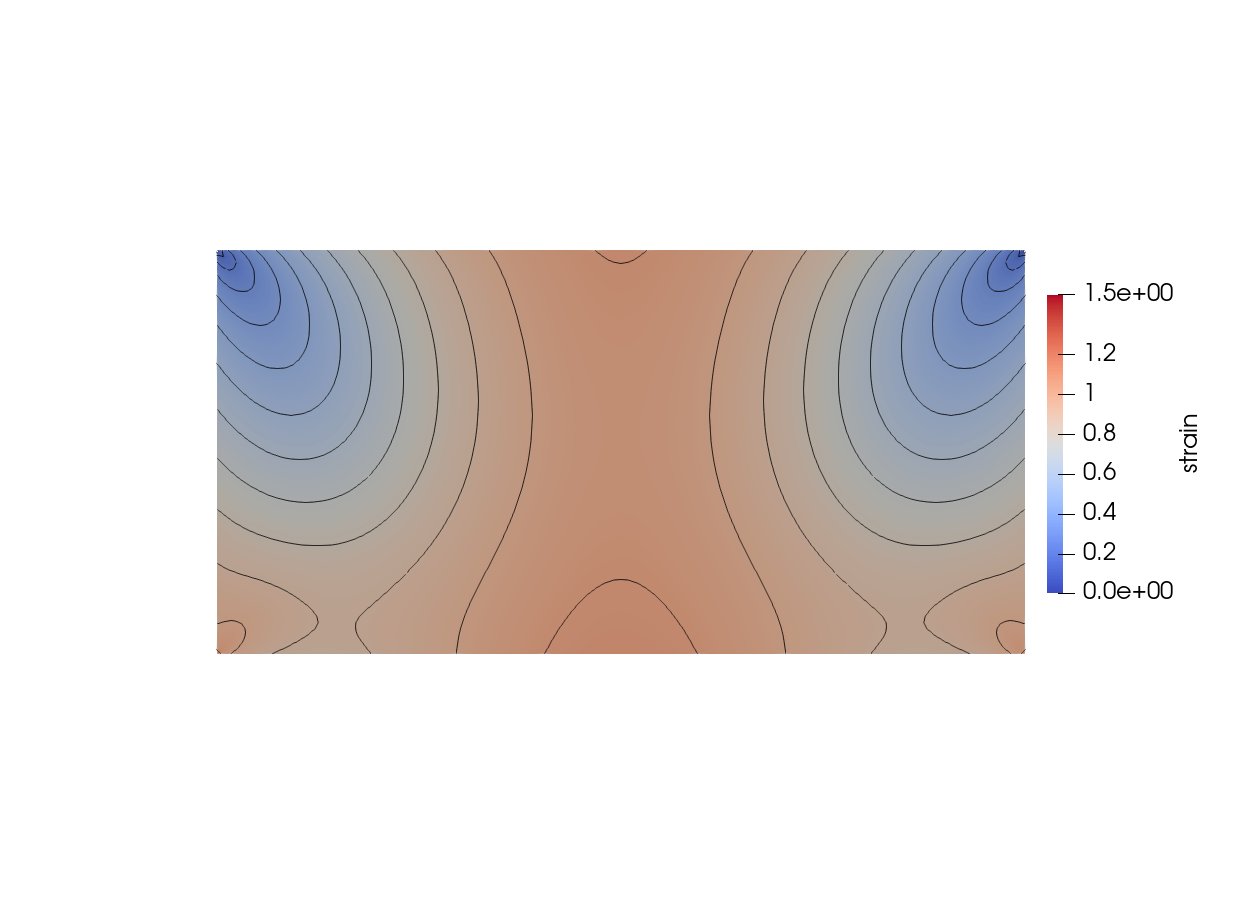} \\    
    \includegraphics[trim={7cm 7cm 7cm 7cm},clip,width=0.325\linewidth]{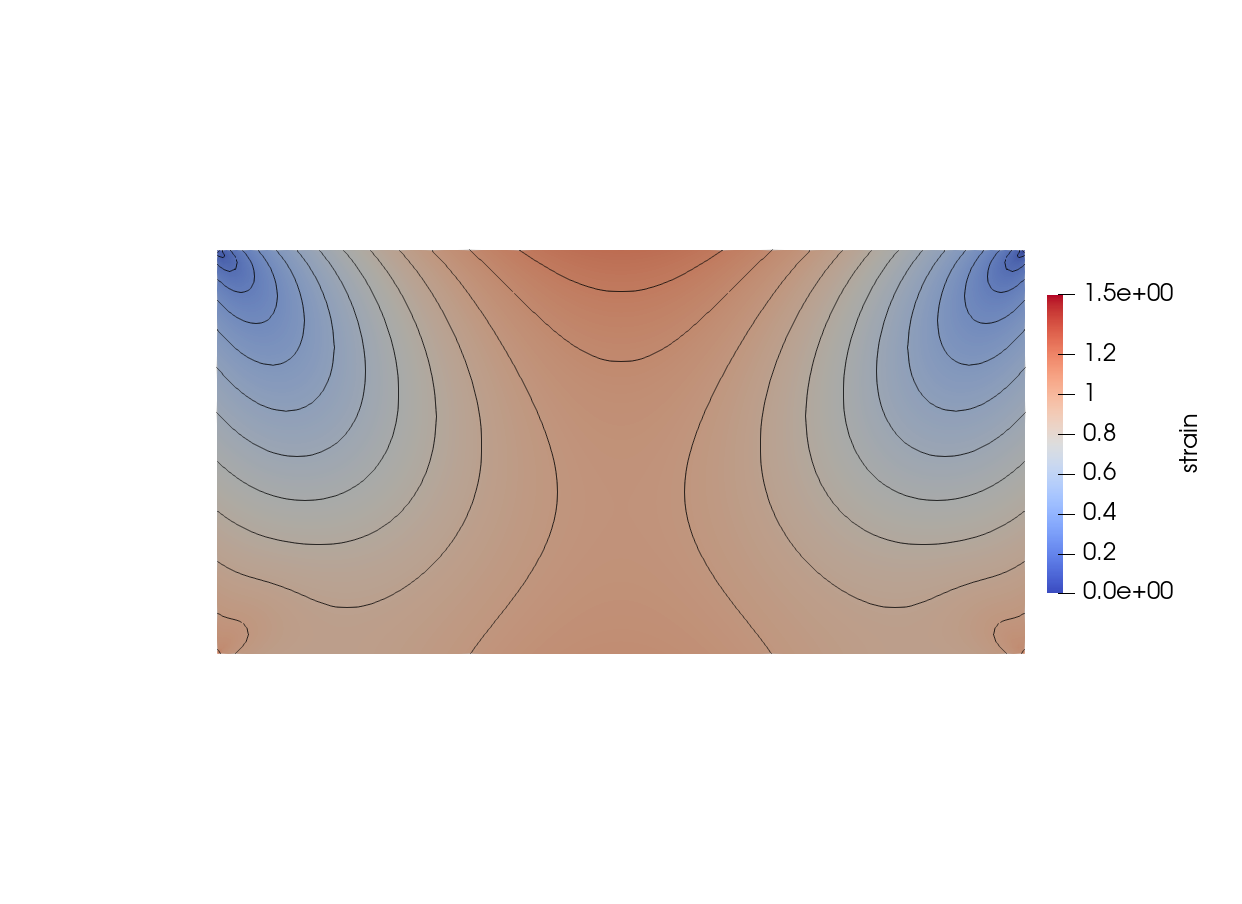}
    \includegraphics[trim={7cm 7cm 7cm 7cm},clip,width=0.325\linewidth]{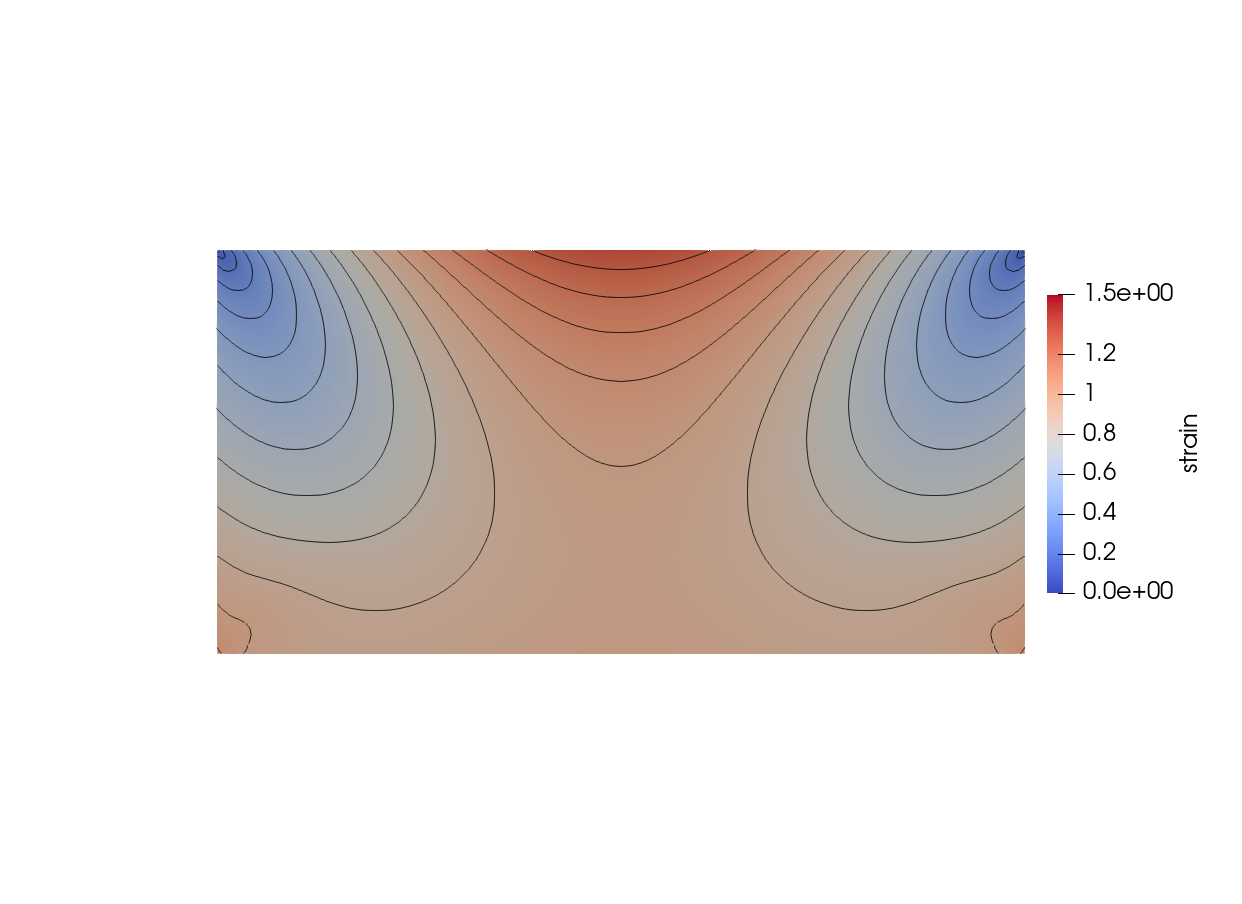}
    \includegraphics[trim={7cm 7cm 7cm 7cm},clip,width=0.325\linewidth]{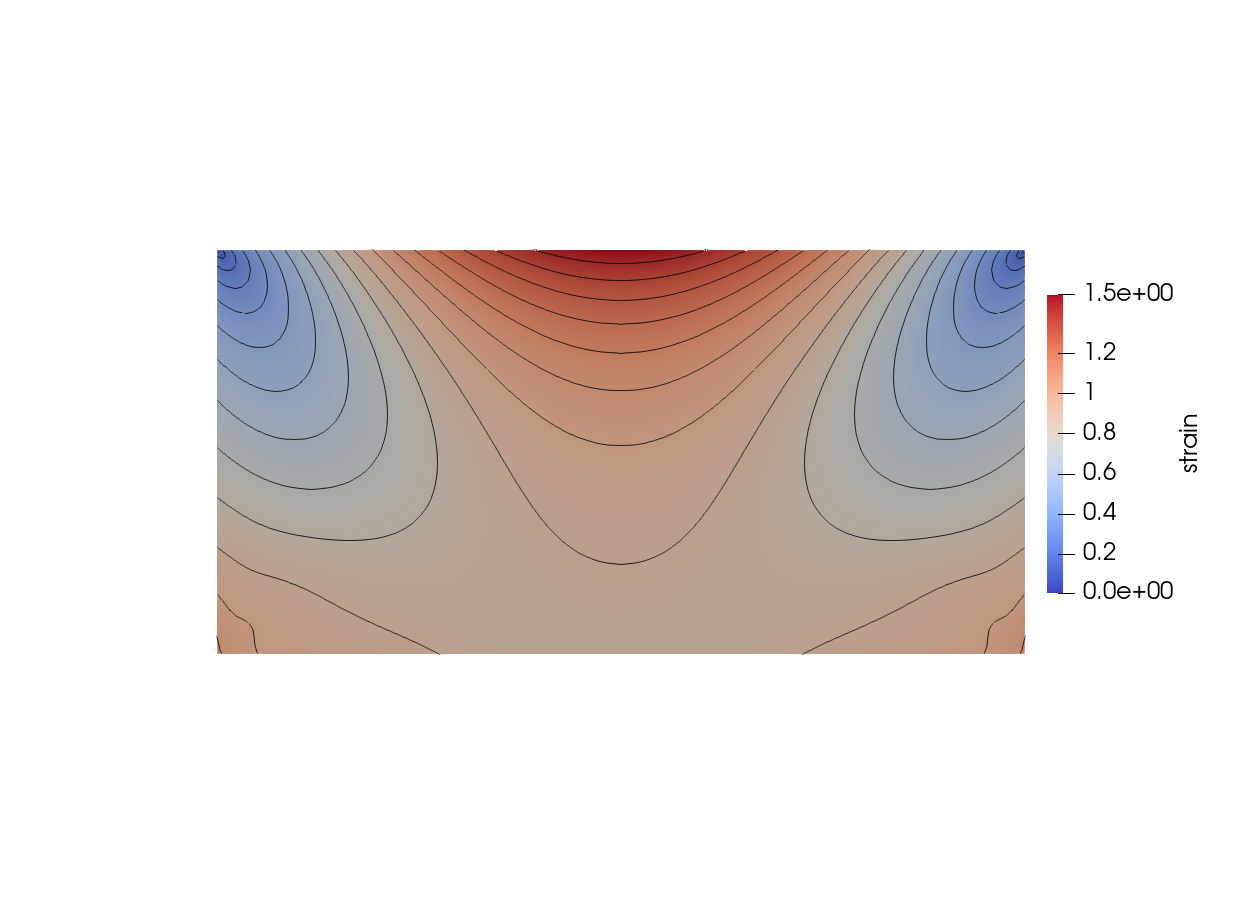}
    \caption{Frobenius norm $|\veps(\vu_h^N)|$ of the strain at the final time $T=1$ when $\lambda_1=\lambda_2=0,1,2,3,4,5$ (from left to right and top to bottom) using 20 values uniformly distributed isolines between $0$ (dark blue) and $1.5$ (dark red).}
    \label{fig:Example2_strain}
\end{figure}

%% Biblio
\bibliographystyle{plain}
\bibliography{bibliography}

\end{document}